\documentclass[a4paper]{article}
\usepackage[utf8]{inputenc}
\usepackage{amsmath,amssymb,amsthm}
\usepackage[colorlinks,citecolor=blue]{hyperref}
\usepackage{cleveref}
\usepackage{mathrsfs}
\usepackage{graphicx} 

\newcommand{\Tamn}{\operatorname{Tam}_n}
\newcommand{\Dyckn}{\operatorname{Dyck}_n}
\newcommand{\Deltamn}{\operatorname{Tam}^\delta_n}

\newcommand{\Y}{\textsf{Y}}

\newtheorem{theorem}{Theorem}[section] 
\newtheorem{proposition}[theorem]{Proposition} 
 
\newtheorem{corollary}[theorem]{Corollary} 
\newtheorem{lemma}[theorem]{Lemma} 
\newtheorem{definition}{Definition}

\newtheorem{remark}[theorem]{Remark}

\title{Linear Intervals in the Tamari and the \\
	Dyck Lattices and in the alt-Tamari Posets}
\author{Clément Chenevière}

\date{August 31, 2022} 

\begin{document}
	
	\maketitle
	
	\begin{abstract} 
		We count the number of linear intervals in the Tamari and the Dyck lattices according to their height, using generating series and Lagrange inversion.
		Surprisingly, these numbers are the same in both lattices.
		We define a new family of posets on Dyck paths, which we call alt-Tamari posets.
		Each alt-Tamari poset depends on the choice of an increment function $\delta \in \{0,1\}^n$.
		We recover the Tamari and the Dyck lattices as extreme cases with $\delta=1$ and $\delta=0$, respectively.
		We prove that all the alt-Tamari posets have the same number of linear intervals of any given height.
		
	\end{abstract}
	
	\tableofcontents
	
	
	\section*{Introduction} 
	
	In a partially ordered set or poset, when two elements $ P $ and $ Q $ are comparable, the interval $ [P,Q] $ is the subset of all elements $ R $ that satisfy $ P \leq R \leq Q $.
	When studying a poset, one may be interested in its intervals and the simplest of them are those which are totally ordered, that we call \emph{linear intervals}.
	
	It is quite natural to count linear intervals when looking at a finite poset, although this has not been much studied.
	In this work, we are interested mainly in two classical posets on Catalan objects, namely the Tamari and the Dyck lattices \cite{tamari,stanley,bernardi-bonichon}.
	Surprisingly, these two posets share the same number of linear intervals of any given height, where the height is the cardinality minus one.
	Furthermore, these numbers are very simple.
	We prove this result and give an explicit formula for these numbers.
	
	This coincidence was a hint that there was a stronger link between these two posets than what was previously known.
	In fact, one can define a whole new family of posets that interpolate between the Tamari and the Dyck lattices, and we call them the alt-Tamari posets.
	We define and study this family in this paper and we also prove that they all have the same number of linear intervals of any fixed height.
	
	There are two other families of posets that generalize the Tamari lattices, namely the Cambrian lattices of type $ A $ and the posets of tilting modules.
	The former were defined by Reading in \cite{reading1,reading2} as the restriction of the weak order on the subset of $ c $-sortable elements for some standard Coxeter element $ c $.
	The latter were defined by Happel and Unger, following work of Riedtmann and Schofield \cite{riedtmann-schofield,happel-unger} in the context of quiver representations. Interested readers may want to look at the survey article \cite{thomas}.
	
	All posets of these two other families also seem to have the same number of linear intervals of any given height as the Tamari lattice.
	
	These three families have even more in common, namely their categories of modules seem to be all derived-equivalent.
	This has been proven for posets of tilting modules and Cambrian posets by Ladkani in  \cite{ladkani1,ladkani2}, but this is only conjectural for the alt-Tamari posets.
	In particular, these new posets between the Tamari and the Dyck lattices might be useful for proving their expected derived equivalence.
	

	In all posets, intervals of height $ 0 $ are the trivial intervals of the form $ [P,P] $ and intervals of height $ 1 $ are those reduced to $ 2 $ elements, that we call the covering relations and that we denote by $ P \triangleleft Q $.
	
	Nevertheless, some posets do not have any linear interval of height $ 2 $ or more, as for instance the third classical poset on Catalan objects, namely the non-crossing partition lattice with coarsening order.
	The Boolean lattice $ B_n $ and the set partition lattice $ P_n $ are two other examples.
	In both cases, all intervals of height $ 2 $ are isomorphic to the Boolean lattice $ B_2 $ or to the set partition lattice $ P_2 $, which are not totally ordered.
	The posets which do not have linear intervals of height $ 2 $ are said to be $ 2 $-thick, and for lattices of finite length, this is equivalent to be relatively complemented \cite{bjorner}.
	
	Let us note that this new family of posets also has close links with the $\nu$-Tamari lattices.
	On the one hand, there are strong hints that all posets and results would generalize to this larger context.
	On the other hand, the alt-Tamari posets can be understood as intervals in some $\nu$-Tamari lattices, as understood after discussions with Wenjie Fang.
	
	In further work, we may look at the Cambrian lattices, the posets of tilting modules and other properties than the number of linear intervals.
	We could also investigate if the techniques in this paper give some results in the case of other posets as for instance the $ m $-Tamari \cite{bergeron-preville-ratelle} and the $ m $-Cambrian lattices \cite{cataland} or the Cambrian lattices and posets of tilting modules in other types.
	Several other families of posets seem to have nice formulas for their numbers of linear intervals, for instance weak orders on the symmetric groups.
	
	
	
	\medskip
	In \cref{sec:def} we recall classical material on intervals and posets and the Tamari and the Dyck lattices.
	In \cref{sec:Tamari,sec:Dyck}, we study the structure of linear intervals in the Tamari and the Dyck lattices.
	In both cases, we first define a notion of left and right intervals, which we prove to be linear.
	We then show that all linear intervals of height $ k \geq 2 $ are either left or right intervals.
	Then, we give a bijection between linear intervals of height $ 1 $ (resp. $ k \geq 2 $) and a marked Catalan object (resp. with a direction) and one Catalan object (resp. a sequence of $ k $ Catalan objects).
	From this bijection, we deduce equations on the generating series of linear intervals of a given height.
	Finally, we solve these equations and compute the number of linear intervals.
	
	In \cref{sec:altTamari}, we define the family of alt-Tamari posets which generalizes both the Tamari and the Dyck lattices.
	We study their linear intervals and show that there are similar notions of left and right intervals and similar decompositions.
	This gives a bijection between linear intervals of any two alt-Tamari posets that respects the height.
	It follows that their linear intervals are counted by the same numbers and this is the main result of this paper.
	\begin{theorem} \label{thm:maintheorem}
		For any increment function of size $ n $, the alt-Tamari poset $ \Deltamn $ contains:
		\begin{itemize}
			\item $ \displaystyle \frac{1}{n+1} \binom{2n}{n} $ linear intervals of height $ 0 $,
			\item  $\displaystyle \binom{2n-1}{n-2}$ linear intervals of height $ 1 $,
			\item  $\displaystyle 2 \binom{2n-k}{n-k-1}$ linear intervals of height $ k $, for $ 2 \leq k < n $.
		\end{itemize}
		Furthermore, there are  no linear interval of height $ k \geq n $.
	\end{theorem}
	
	We also prove additional results on the family and in particular that there is a boolean structure of refinement of the alt-Tamari posets.

	\section{Definitions and first properties}  \label{sec:def}
	
	\subsection{Linear intervals}
	In a poset, if we have two elements $ P \leq Q $, then a \emph{chain} of length $ k $ from $ P $ to $ Q $ is a sequence of elements $ P = P_0 < P_1 < \dots < P_k = Q $.
	Recall that the \emph{height} of an interval $ [P,Q] $ is defined as the maximal length of a chain from $ P $ to $ Q $ and that an interval is \emph{linear} if it is totally ordered.
	
	In any poset, the intervals of height $ 0 $ are those of the form $ [P,P] $ with $ P $ some element of the poset.
	They are clearly linear.
	We call them \emph{trivial intervals} and they obviously are in bijection with elements of the poset.
	
	Intervals $ [P,Q] $ of height $ 1 $ are by definition reduced to two elements since otherwise they would contain a chain of length $ 2 $ or more.
	They are linear as well.
	Such an interval is called a \emph{covering relation} and is denoted $ P \triangleleft Q $.
	
	A chain is said to be \emph{maximal} if for all $ 0 \leq i < k $, $ P_i \triangleleft P_{i+1} $ is a covering relation.
	In a finite poset, an interval $ [P,Q] $ is linear if and only if there is a unique maximal chain from $ P $ to $ Q $.
	The length of this maximal chain is exactly the \emph{height} of the interval.
	
	We can define the Cartesian product of two posets $ P $ and $ Q $ as a poset $ P \times Q $ whose elements are couples $ (p,q) $ with $ p \in P $ and $ q \in Q $ and whose order is $ (p,q) \leq (p',q') $ component wise.
	The Cartesian product of two intervals in $ P $ and $ Q $ gives obviously an interval in $ P \times Q $, but the converse is also true, namely every interval of $ P \times Q $ comes from the Cartesian product of two intervals.
	
	We can thus keep track of the counting of linear intervals in a multiplicative way in the Cartesian product of two posets.
	To do so, we have to distinguish trivial intervals and linear intervals which are not trivial.
	Indeed, the Cartesian product of two intervals is linear if and only if one of the intervals is trivial and the other is linear, and it is trivial if and only if both intervals are trivial.
	
	Let $ \epsilon $ be a formal variable with $\epsilon^2 = 0$.
	Given a poset $ P $, we can define the polynomial $ L(P) = T(P) + U(P) \epsilon $ where $ T(P) $ is the number of trivial intervals of $ P $ and $ U(P) $ counts its non trivial linear intervals.
	Then, we have  \begin{equation*}
		L(P \times Q) = L(P)  L(Q)
	\end{equation*} in $ \mathbb{Z}[\epsilon] $.
	This procedure could be helpful in generalizing some results or in counting linear intervals in posets where intervals are understood through a recursive decomposition.

	\subsection{Trees and Tamari lattices}
	
	
	
	A \emph{planar rooted binary tree} is a finite connected acyclic planar graph whose vertices have degree $ 3  $ or $ 1 $, with one marked vertex of degree $ 1 $ called the root.
	The other vertices of degree $ 1 $ are called the leaves and those of degree $ 3 $ are the nodes.
	In what follow all trees will be planar rooted binary trees and we will draw them with leaves at the top and the root at the bottom.
	The node connected to the root will be called the root node.
	We will usually not draw the root and have thus only the root node of degree $ 2 $ at the bottom.
	The size of a tree will be its number of nodes.
	The tree of size $ 0 $, reduced to a leaf will be considered as a trivial tree.
	
	For $n \geq 1$, let us denote by $ Y_n $ the set of all trees of size $n$.
	The Tamari lattice $\Tamn$ of size $n$ is a partial order on $ Y_n $, described as the reflexive transitive closure of the left rotations (see \cref{fig:coveringrelation}) .
	This poset was first defined and studied by Tamari \cite{tamari} and happens to be a lattice, see \cite{tamari-huang}.
	Its covering relations are exactly all left rotations.
	We will usually not consider the Tamari lattice of size $ 0 $ and its unique trivial interval.

	We can \emph{graft} a tree $ T' $ on a chosen leaf of another tree $ T $.
	This is done by deleting the root of $ T' $ and identifying the root node $ r $ of the tree $ T' $ with the chosen leaf of $ T $, producing a tree whose size is the sum of the sizes of $ T $ and $ T' $, as drawn in \cref{fig:grafting}.
	Grafting a tree of size $ 0 $ does nothing.
	
	\begin{figure}[h]
		\centering
		\includegraphics[width=5cm]{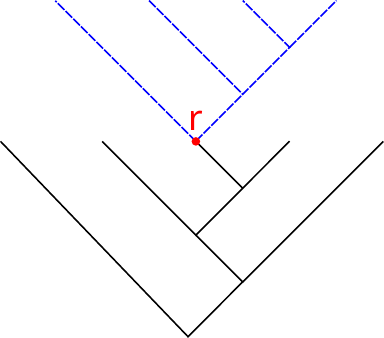}
		\caption{Grafting a tree on another one.}
		\label{fig:grafting}
	\end{figure}
	
	We can also \emph{plug} a tree $ T' $ into a chosen edge of another tree $ T $, except for the edge between the root and the root node.
	To do so, we create a new node $ n $ on the selected edge of $ T $, and we identify this node with the root of $ T' $.
	If the selected edge was a left son, the tree $ T' $ will be the right son of $ n $, and vice versa.
	If $ T $ and $ T' $ are respectively of size $ m $ and $ m' $, plugging $ T' $ into an edge of $ T $ will produce a tree of size $ m+m'+1 $.
	In \cref{fig:plugging}, we take the same two trees $ T $ and $ T' $ as in the example of grafting, but this time, we plug $ T' $ into the red dotted edge of $ T $.
	Note that plugging a tree of size $ 0 $ on an edge of a tree $ T $ produces a tree that is different from $ T $, with one more vertex.
	
	\begin{figure}[h]
		\centering
		\includegraphics[width=5cm]{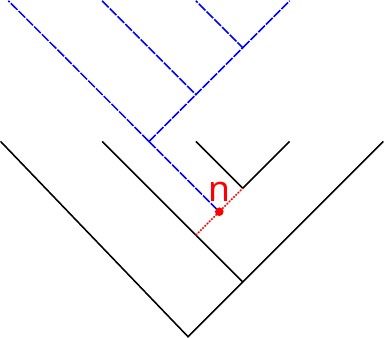}
		\caption{Plugging a tree into the selected (red dotted) edge of another one.}
		\label{fig:plugging}
	\end{figure}
	
	In the following, we will need a special family of trees, namely left and right combs, recursively defined as follows:
	$ \ell_1 = r_1 =: \Y$ is the only tree of size $ 1 $ and for $ n \geq 2 $, $ \ell_n $ (resp. $ r_n $) is obtained as the grafting of $ \Y $ on the leftmost (resp. rightmost) leaf of $ \ell_{n-1} $ (resp. $ r_{n-1} $).
	In \cref*{fig:R3L4}, the first tree of the first row is the right comb $ r_4 $ and the last tree of the second row is the left comb $ \ell_5 $.
	
	There is an involution on trees that exchanges a tree with its mirror image.
	This can be seen as a vertical symmetry on the drawings.
	As an example, the right comb $ r_n $ and the left comb $ \ell_n $ are exchanged through this involution.
	
	\begin{remark}
		This involution on trees is an order-reversing involution on the Tamari lattice.
	\end{remark}
	\begin{proof}
		Let  $ P \triangleleft Q $ be a covering relation and $ P^* $ (resp. $ Q^* $) be the mirror image of $ P $ (resp. $ Q $).
		Then $ Q^* \triangleleft P^* $ is a covering relation in the Tamari lattice.
	\end{proof}
	
	\subsection{Structure of intervals in the Tamari lattices}
	
	In \cite{chapoton1}, the author introduces an operadic structure on intervals in the Tamari lattices with the operation of grafting of intervals.
	See this article for more details on the results of this section.
	
	If $ I = [P,Q] $ is an interval in $ \operatorname{Tam}_n $ and $ I' = [P',Q'] $ is an interval in $ \operatorname{Tam}_m $, we can define the grafting of $ I' $ on the $ k $-th leaf of $ I $ for some $ 0 \leq k \leq n $, producing an interval $ I'' \in \operatorname{Tam}_{n+m}$ as represented in \cref{fig:graftinterval}.
	The bottom element of $ I'' $ is obtained by grafting $ P' $ on the $ k $-th leaf of $ P $ and the top element of $ I'' $ is obtained by grafting $ Q' $ on the $ k $-th leaf of $ Q $.
	As an interval, $ I'' $ is isomorphic to the product of the intervals $ I $ and $ I' $.
	
	\begin{figure}[h] 
		\begin{center}
			\centering
			\begin{minipage}[c]{.4\linewidth}
				\begin{center}
					\includegraphics[width=4.5cm]{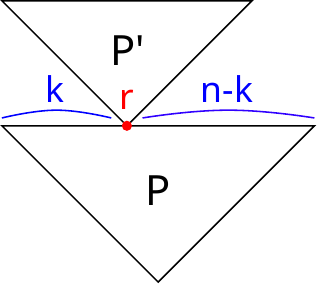}
				\end{center}
			\end{minipage}
			$ \leq $
			\begin{minipage}[c]{.4\linewidth}
				\begin{center}
					\includegraphics[width=4.5cm]{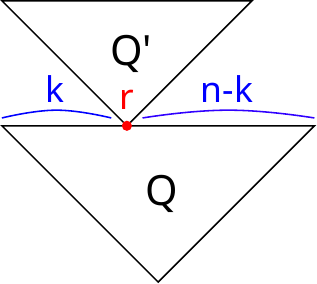}
				\end{center}
			\end{minipage}
			\caption{Grafting an interval $ [P',Q'] $ on the $ k $-th leaf of the interval $ [P,Q] $.}
			
			\label{fig:graftinterval}
		\end{center}
	\end{figure} 
	
	The author also defines \emph{new intervals}, that cannot be written as an interval grafted on another one. 
	This gives a unique decomposition of intervals into new intervals.
	Any interval then has the structure of the product of the new intervals in its decomposition.
	
	\begin{remark} \label{rem:grafting}
		If an interval $ I'' = [P'',Q''] $ is not new, then it decomposes into an interval $ I' = [P',Q'] $ grafted on the $ k $-th leaf of some interval $I = [P,Q]  \in \Tamn $ for some $ 0 \leq k \leq n $.
		
		Thus, there is a common node $ r $ in $ P'' $ and $ Q'' $ such that there are $ k $ leaves on the left of $ r $ and $ n-k $ leaves on its right, as show in \cref{fig:graftinterval}.
		The node $ r $ is precisely the root node of $ P' $ and $ Q' $ that has been identified to the selected leaf of $ P $ and $ Q $.
	\end{remark}
	
	In fact, such common nodes between the bottom and top elements of an interval are precisely the root nodes of the new intervals in its decomposition.
	
	\subsection{Dyck paths and Dyck lattices} \label{sec:Dyckpaths}
	
	A \emph{Dyck path} of length (or size) $n$ is a path in $\mathbb{N}^2$ consisting of up steps $(1,1)$ and down steps $(1,-1)$, starting from $(0,0)$ and ending at $(2n,0)$. 
	A Dyck path can be seen as a word in letters $ u $ (for up steps) and $ d $ (for down steps) with each letter appearing $ n $ times and such that in any prefix the letter $ u $ appears at least as many times as the letter $ d $.
	Such words are also called Dyck words and we will identify a Dyck path to its Dyck word, and more generally any word in letters $ u $ and $ d $ with the associated path.
	The height of a step is defined as the second coordinate of the vertex at which it starts.
	Equivalently, it is equal to the number of $ u $ preceding it minus the number of $ d $ preceding it.
	
	In all of the following, a subword $ B $  of a word $ A $ will be a portion of consecutive letters of $ A $.
	We will denote $ B \subset A $.
	
	There is a natural matching of up and down steps. 
	Given an up step, we can define the \emph{excursion} starting at this step as the portion of the path that starts at this up step and ends at the first down step that ends at the same height.
	In other words, it is the smallest subword starting at this letter $ u $ that is a Dyck word.
	Then, the last step of the excursion is the down step that matches the first step of the excursion.
	
	For $ n \geq 1 $, let us denote by $ Z_n $ the set of all Dyck paths (or words) of size $ n $.
	A natural order on $ Z_n $ is the inclusion order, where we say that a path $ P $ is included in a path $ Q $ if $ P $ is always under $ Q $ when we draw them together.
	In other words, for every $ 1 \leq k \leq 2n $, the height of the $ k $-th step of $ P $ is less than or equal to the height of the $ k $-step of $ Q $. 
	This produces the \emph{Dyck lattice} $ \Dyckn $ of size $ n $, sometimes known as the Stanley lattice (see \cite{stanley,bernardi-bonichon}). 
	
	Given a Dyck path, we define a \emph{valley} as a down step followed by an up step, and a \emph{peak} as an up step followed by a down step.
	The covering relations in $ \Dyckn $ can be defined as transforming a valley $ du $ in a peak $ ud $.
	Similarly as in the Tamari lattice, the path of size $ 0 $ will be considered as a trivial path but we will usually not consider the Dyck lattice of size $ 0 $ and its unique trivial interval.
	
	\begin{remark}
		We can define another order on the Dyck paths, where the covering relations consist of swapping a down step with the excursion that follows it, if any.
		
		When doing so, we recover a lattice isomorphic to the Tamari lattice that we defined on trees.
		This alternative description was introduced by Bergeron and Préville-Ratelle in \cite{bergeron-preville-ratelle}.
		
		Furthermore, the inclusion order refines the Tamari order, in the sense that any interval $ [P,Q] $ in the Tamari lattice defined this way will be an interval in $ \Dyckn $.
	\end{remark}
	
	We will prove that both lattices have the same number of linear intervals of any height, however a linear interval in $ \Tamn $ will usually not be linear in $ \Dyckn $.
	
	There is also an involution on Dyck paths that exchanges a path with its mirror image, that can be seen as a vertical symmetry on the drawings.
	More precisely, it reverses up and down steps and the order of the steps.
	
	\begin{remark}
		This involution on Dyck paths is an order-preserving involution on the Dyck lattice.
		However, it is neither an order-preserving nor an order-reversing involution on the Tamari lattice.
	\end{remark}

	\section{Linear intervals in the Tamari lattices} \label{sec:Tamari}

	\subsection{Structure of linear intervals}

	Let us define two particular intervals $ L_n $ and $ R_n $ in $ \operatorname{Tam}_{n+1} $, for $ n \geq 2$.
	We will see that they are actually the only linear new intervals in $  \operatorname{Tam}_{n+1}$ and also its linear  intervals of greatest height, namely $ n $.
	
	The interval $ R_n $ has the right comb $ r_{n+1} $ as bottom element and its top element is obtained from $ r_{n+1} $ by performing one rotation at each node of the right side, from top to bottom.
	In other words, the top element of $ R_n $ is the tree whose root node has a leaf as right son and a right comb $ r_n $ as left son.
	
	The interval $ L_n $ is the mirrored version of $ R_n $.
	Its top element is the left comb $ \ell_{n+1} $ and its bottom element is the tree whose root node has a leaf as left son and a left comb $ \ell_n $ as right son.
	
	\begin{remark} \label{rem:newlinear}
		For $ n \geq 2 $, both intervals $ L_n $ and $ R_n $ are linear of height $ n $ since there is only one possible way to perform a right rotation to go down in $ R_n $ at every step (or a left rotation to go up in $ L_n $).
		
		Moreover, they are new since their top and bottom element do not share any common node as explained in \cref{rem:grafting}. Indeed, all nodes in the bottom element of $ L_n $ have exactly one leaf strictly on their left whereas all nodes in the top element of $ L_n $ have no leaves strictly on their left.
		
	\end{remark}
	We can also define $ L_1 = R_1 $ as the interval whose bottom element is $ r_2 $ and whose top element is $ \ell_2 $.
	It is linear of height $ 1 $ and new. 
	
	\begin{figure}[h]
		\begin{center}
			\centering
			\begin{minipage}[c]{.2\linewidth}
				\includegraphics[width=2.5cm]{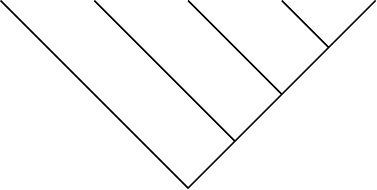}
			\end{minipage}
			$ \triangleleft $
			\begin{minipage}[c]{.2\linewidth}
				\includegraphics[width=2.5cm]{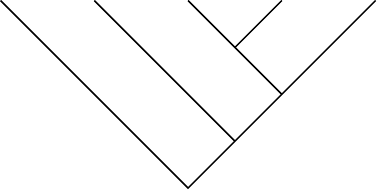}
			\end{minipage}
			$ \triangleleft $
			\begin{minipage}[c]{.2\linewidth}
				\includegraphics[width=2.5cm]{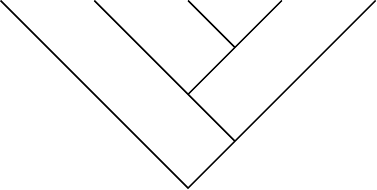}
			\end{minipage}
			$ \triangleleft $
			\begin{minipage}[c]{.2\linewidth}
				\includegraphics[width=2.5cm]{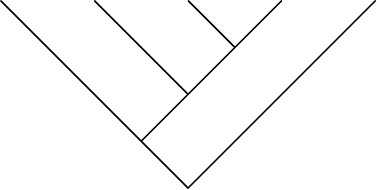}	
			\end{minipage}
			
			\bigbreak
			
			\begin{minipage}[c]{.4\linewidth}
				\begin{center}
					\includegraphics[width=3.5cm]{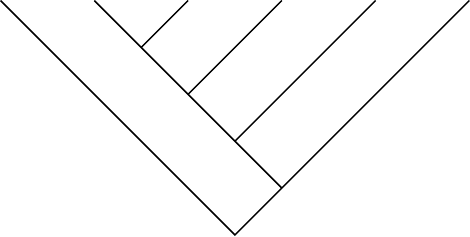}
				\end{center}
			\end{minipage}
			$ \leq $
			\begin{minipage}[c]{.4\linewidth}
				\begin{center}
					\includegraphics[width=3.5cm]{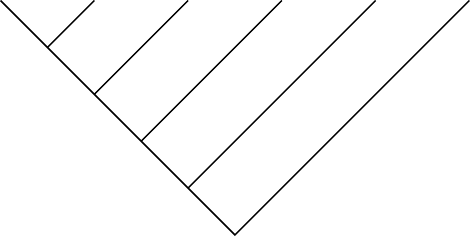}
				\end{center}
			\end{minipage}
			\caption{The intervals $ R_3 $ with all $ 4 $ elements (top) and $ L_4 $ (bottom).} 
			\label{fig:R3L4}
		\end{center}
	\end{figure}

	
	We now prove that there are four kinds of linear intervals in the Tamari lattice, namely trivial intervals, covering relations and "left" or "right" intervals.
	Left (resp. right) intervals are defined by first grafting trivial intervals on the leaves of an interval $ L_n $ (resp. $ R_n $) for some $ n \geq 2 $ and then grafting the result on some leaf of a trivial interval.
	As an interval, it is the product of trivial intervals and one linear interval, and thus, it is linear.
	It remains to prove that all linear intervals are of this form.
	
	\begin{figure}[h]
		\begin{center}
			\centering
			\begin{minipage}[c]{.28\linewidth}
				\begin{center}
					\includegraphics[width=3.5cm]{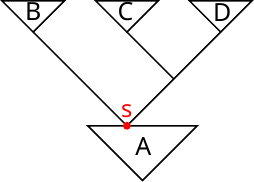}
				\end{center}
			\end{minipage}
			$ \triangleleft $
			\begin{minipage}[c]{.28\linewidth}
				\begin{center}
					\includegraphics[width=3.5cm]{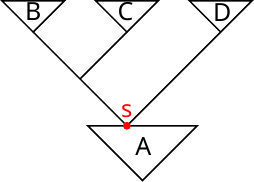}
				\end{center}
			\end{minipage}
			\caption{A covering relation.}
			\label{fig:coveringrelation}
		\end{center}
	\end{figure} 
	
	\begin{figure}[h]
		\begin{center}
			\centering
			\begin{minipage}[c]{.4\linewidth}
				\begin{center}
					\includegraphics[width=4.5cm]{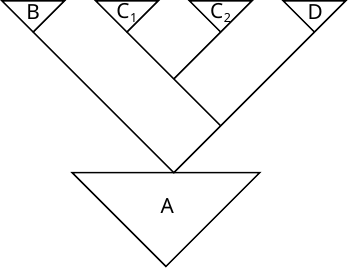}
				\end{center}
			\end{minipage}
			$ \leq $
			\begin{minipage}[c]{.4\linewidth}
				\begin{center}
					\includegraphics[width=4.5cm]{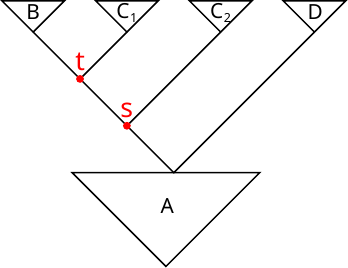}
				\end{center}
			\end{minipage}
			\caption{A left interval of height $ 2 $.}
			\label{fig:leftinterval}
		\end{center}
	\end{figure} 
	
	\begin{figure}[h]
		\begin{center}
			\centering
			\begin{minipage}[c]{.4\linewidth}
				\begin{center}
					\includegraphics[width=5cm]{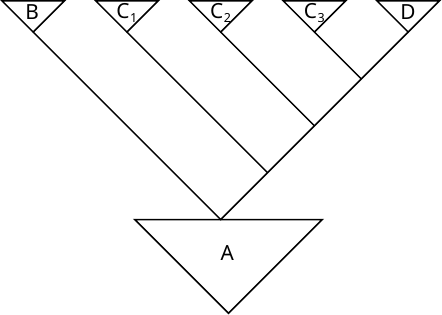}
				\end{center}
			\end{minipage}
			$ \leq $
			\begin{minipage}[c]{.4\linewidth}
				\begin{center}
					\includegraphics[width=5cm]{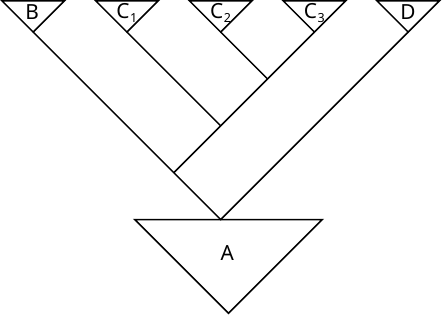}
				\end{center}
			\end{minipage}
			
			\caption{A right interval of height $ 3 $.}
			\label{fig:rightinterval}
		\end{center}
	\end{figure} 
	
	
	\begin{proposition}
		All linear intervals of height $ k \geq 2 $ are either left or right intervals.
	\end{proposition}
	
	\begin{proof}
		We prove the result for $ k = 2 $ and then by induction.
		
		Suppose we have a linear interval of height $ 2 $, that is of the form $ P \triangleleft Q \triangleleft Q' $.
		Then we know that $ P \triangleleft Q $ is a covering relation at some node $ s $ as in \cref{fig:coveringrelation}, for some trees $ A, B, C, D $.
		More precisely, $ P $ (resp. $ Q $) can be described as a comb $ r_2 $ (resp. $ \ell_2 $) grafted on the $ k $-th leaf of a tree $ A $, with the trees $ B, C, D $ grafted from left to right on the leaves of $ r_2 $ (resp. $ \ell_2 $).
		
		Let us study the next possible rotations for the covering relation $ Q \triangleleft Q' $.
		As we can see in \cref{fig:square}, if we perform a rotation within $ A, B, C $ or $ D $, then we get a square, thus, not a linear interval.
		More precisely, if the next rotation $ Q \triangleleft Q' $ is such that $ Q' $ can be described as a comb $ \ell_2 $ grafted on the $ k $-th leaf of a tree $ A' $, with the trees $ B', C', D' $ grafted from left to right on the leaves of $ \ell_2 $, then we can define a tree $ P' $ as $ Q' $ where the left comb $ \ell_2 $ is replaced by a right comb $ r_2 $.
		Then we have $ P \triangleleft P' \triangleleft Q' $ and $ [P, Q'] $ is not linear.
		
		This being excluded, there remains at most three possible nodes for a rotation as in \cref{fig:possiblenodes}, namely the same node $ s $, its left successor $ t $ (if any), and its predecessor $ u $ (if $ s $ is a right son).
		
		If we perform another rotation at the node $ s $, we get a pentagon as in \cref{fig:pentagon}, thus not a linear interval.
		
		Otherwise, a rotation at the node $ t $ produces a left interval of height $ 2 $ as in \cref{fig:leftinterval} and a rotation at the node $ u $ produces a right interval of height $ 2 $ similar to the one in \cref{fig:rightinterval}, both being linear intervals.
		
		\medskip
		
		Let now $ [P,Q] $ be a linear interval of height $ k \geq 3 $. Let $ Q' $ be the lower cover of $ Q $ in $ [P,Q] $, such that $ [P,Q'] $ is a linear interval of height $ k-1 \geq 2 $.
		By induction, it is either a left of a right interval.
		
		\smallskip
		Assume $ [P,Q'] $ is a left interval.
		Then, the last rotation in $ [P,Q'] $ happens at a node $ s $ as in \cref{fig:leftinterval}.
		The node $ s $ is not a right son, so the only possible rotation to get a linear interval is at its left successor $ t $.
		Thus, $ [P,Q] $ is a left interval of height $ k $.
		
		\smallskip
		Assume $ [P,Q'] $ is a right interval, and the last rotation happens at a node $ s $, then a rotation to its left son (if possible) would produce a non linear interval as in \cref{fig:pentagon}.
		The only possible rotation is thus at the predecessor of $ s $ and $ [P,Q] $ is a right interval of height $ k $.
	\end{proof}
	
	\begin{figure}[h]
		\begin{center}
			\centering
			\begin{minipage}[c]{.4\linewidth}
				\begin{center}
					\includegraphics[width=4cm]{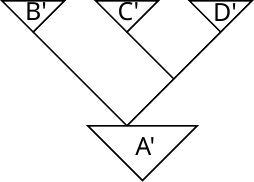}
				\end{center}
			\end{minipage}
			$ \leq $
			\begin{minipage}[c]{.4\linewidth}
				\begin{center}
					\includegraphics[width=4cm]{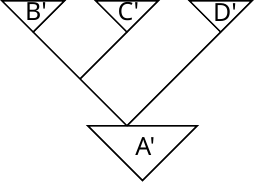}
				\end{center}
			\end{minipage}
			
			\bigbreak
			
			\begin{minipage}[c]{.45\linewidth}
				\begin{center}
					\rotatebox[origin=c]{90}{$\leq$}
				\end{center}
			\end{minipage}
			\begin{minipage}[c]{.45\linewidth}
				\begin{center}
					\rotatebox[origin=c]{90}{$\leq$}
				\end{center}
			\end{minipage}
			\bigbreak
			\begin{minipage}[c]{.4\linewidth}
				\begin{center}
					\includegraphics[width=4cm]{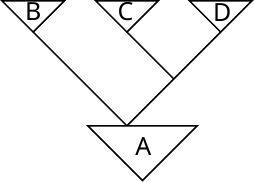}
				\end{center}
			\end{minipage}
			$ \leq $
			\begin{minipage}[c]{.4\linewidth}
				\begin{center}
					\includegraphics[width=4cm]{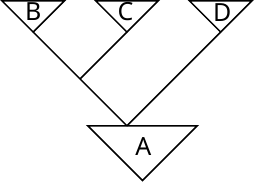}
				\end{center}
			\end{minipage}
			
			\caption{Non linear interval.}
			\label{fig:square}
		\end{center}
	\end{figure} 
	
	\begin{figure}[h]
		\centering
		\includegraphics[width=5cm]{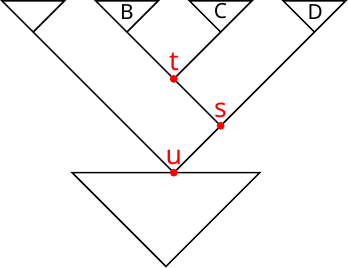}
		\caption{The tree $ Q $ after rotation at $ s $.}
		\label{fig:possiblenodes}
	\end{figure}
	
	\begin{figure}[h]
		\begin{center}
			\centering
			\begin{minipage}[c]{.3\linewidth}
				\begin{center}
					\includegraphics[width=3cm]{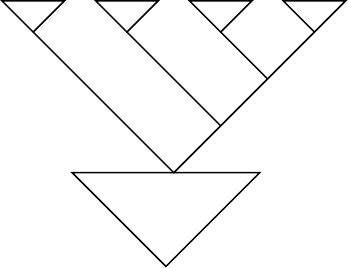}
				\end{center}
			\end{minipage}
			$ \leq $
			\begin{minipage}[c]{.3\linewidth}
				\begin{center}
					\includegraphics[width=3cm]{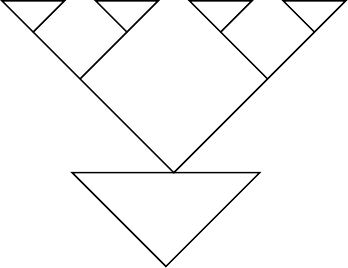}
				\end{center}
			\end{minipage}
			$ \leq $
			\begin{minipage}[c]{.3\linewidth}
				\begin{center}
					\includegraphics[width=3cm]{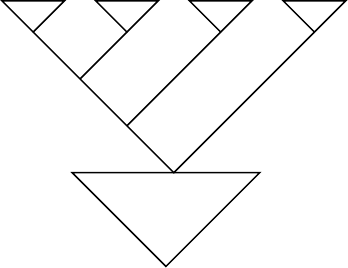}
				\end{center}
			\end{minipage}
			
			\bigbreak
			
			\begin{minipage}[c]{.45\linewidth}
				\begin{center}
					\rotatebox[origin=c]{-45}{$\leq$}
				\end{center}
			\end{minipage}
			\begin{minipage}[c]{.45\linewidth}
				\begin{center}
					\rotatebox[origin=c]{45}{$\leq$}
				\end{center}
			\end{minipage}
			
			\bigbreak
			\begin{minipage}[c]{.3\linewidth}
				\begin{center}
					\includegraphics[width=3cm]{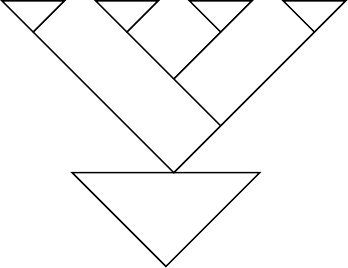}
				\end{center}
			\end{minipage}
			$ \leq $
			\begin{minipage}[c]{.3\linewidth}
				\begin{center}
					\includegraphics[width=3cm]{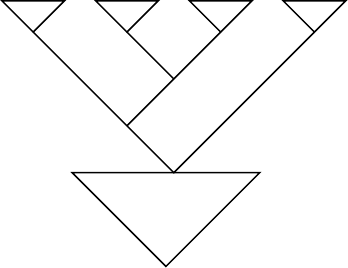}
				\end{center}
			\end{minipage}
			
			\caption{Two consecutive rotations at the same node produce a non linear interval.}
			\label{fig:pentagon}
		\end{center}
	\end{figure} 
	
	We have proven that all linear intervals of height $ k \geq 1 $ decompose as an interval $ L_k $ or $ R_k $, grafted on some tree $ A $ with some trees grafted on the leaves of $ L_k $ or $ R_k $.
	Thus, as long as $ A $ or one of the trees we grafted on $ L_k $ or $ R_k $ is a non trivial tree, we get an interval which is not new.
	Consequently, $ L_n $ and $ R_n $ are the only linear intervals that are new in $  \operatorname{Tam}_{n+1}$.
	
	\begin{remark} \label{rem:maxheightlinear}
		A linear interval of height $ k \geq 1 $ has at least $ k+1 $ nodes.
		Thus, for $ n \geq 1 $ there are no linear interval of height $ k \geq n $ in $ \Tamn $.
	\end{remark}
	
	\subsection{Combinatorial description of linear intervals}
	
	We have described the structure of all linear intervals according to their height.
	We now deduce a combinatorial description of them in order to count them.
	
	A linear interval of height $ 0 $ is of the form $ [P,P] $ where $ P $ is a tree not reduced to a leaf.
	
	A linear interval of height $ 1 $ is a covering relation.
	As we can see in \cref{fig:coveringrelation}, the part with $ A, B $ and $ D $ can be understood as a tree with a marked node $ s $, together with another tree $ C $ that we plug into the right edge out of $ s $ for the bottom element and into the left edge out of $ s $ for the top element. Note that $ A, B, C $ and $ D $ might be just leaves.
	
	Similarly, as in \cref{fig:rightinterval,fig:leftinterval}, a linear interval of height $ k \geq 2 $ can be understood as a tree formed of $ A, B $ and $ D $, with a marked node $ s $ and a direction (left or right), together with a sequence of $ k $ trees $ C_1, \dots, C_k $.
	Those $ k $ trees are then plugged into the selected edge going out of $ s $ or on a common branch then plugged into the other edge out of $ s $.
	
	This enables us to write down equations on the generating series of linear intervals.
	Let $ A(t) $ be the generating series of trees, including the degenerate case of the tree reduced to a leaf.
	The standard decomposition of the trees along the left side for instance gives the equation $ A = 1 + tA^2 $, and the number of trees with $ n $ nodes is the Catalan number $ \displaystyle C_n = \frac{1}{n+1} \binom{2n}{n} $. Moreover, the generating series of trees with a marked node is equal to $ tA'(t) $ and there are $ \displaystyle n C_n = \frac{n}{n+1} \binom{2n}{n} $ of them.
	
	For $ k \geq 0 $, let $ S_k(t) $ be the generating series of linear intervals of height $ k $.
	Since we did not consider the tree reduced to a leaf for the Tamari lattice, we have:
	\begin{equation*}
		S_0 = A - 1.
	\end{equation*}
	Then, a covering relation is given by a marked tree and another tree that we connect to it, which creates an additional node, so we can write:
	\begin{equation*}
		S_1 = (tA')  (tA) = t^2 A' A.
	\end{equation*}
	Similarly, for $ k \geq 2 $, we have the following equation:
	\begin{equation*}
		S_k =  (2tA') (tA)^k.
	\end{equation*}
	Here, the factor $ 2 $ is for the choice of the direction, and we plug $ k $ trees, creating one node for each of them.
	
	\subsection{Counting of linear intervals}
	
	In this section, we use Lagrange inversion \cite[ch. 5]{stanley} to get the coefficients of $ S_k $. Let us introduce the series $ B = A - 1 $ of non trivial trees. We can then write $ B = t (B+1)^2 = t F(B)$, where $ F(x) = (x+1)^2 $.     
	Then, it follows that
	
	\begin{equation*}
		B' = \dfrac{(B+1)^3 }{1-B}.
	\end{equation*}
	
	For $ k \geq 2 $, we can write: 
	\begin{equation*}
		S_1 = t^2 B' (B+1) = t^2 \phi_1(B) \quad \text{and} \quad
		S_k = 2 t^{k+1} B' (B+1)^k = 2 t^{k+1} \phi_k(B),
	\end{equation*} 
	where $ \phi_k(x) = \dfrac{(1+x)^{k+3}}{1-x}  $ for $  k \geq 1 $.
	
	In what follows, we use the notation $ [t^n] S $ to denote the coefficient of degree $ n $ in the series $ S(t) $.
	As $ \phi_k(x) $ is a series in $ x $ with a constant term equals to $ 1 $ and $ B(t) $ is a series in $ t $ with no constant term, we have $ [t^0] \phi_k(B) = 1$.
	
	We can write $\displaystyle \dfrac{1}{1-x} = \sum\limits_{k \geq 0} x^k$ and $\displaystyle \dfrac{1}{(1-x)^2} = \sum\limits_{k \geq 0} (k+1)x^k$.
	Let us compute $ [t^n] \phi_k(B)$, using Lagrange inversion, for $ n \geq 1 $:
	
	\begin{align*}
		[t^n] \phi_k(B) & = \frac{1}{n} [x^{n-1}] \phi_k'(x) F(x)^n\\
		& = \dfrac{k+3}{n} [x^{n-1}] \dfrac{(1+x)^{k+2+2n}}{1-x}  + \dfrac{1}{n}[x^{n-1}] \dfrac{(1+x)^{k+3+2n}}{(1-x)^2}\\
		& = \dfrac{k+3}{n} \sum_{j=0}^{n-1} \binom{k+2+2n}{j}  + \dfrac{1}{n} \sum_{j=0}^{n-1} \binom{k+3+2n}{j} (n-j-1+1).
	\end{align*}
	After a few lines of computation, we obtain that $ \displaystyle [t^n] \phi_k(B) = \binom{k+2+2n}{n} $.
	
	We can notice that this formula is still true for $ n = 0 $.
	
	\begin{theorem}
		In the Tamari lattice $ \Tamn $ of size $ n $, there are:
		\begin{itemize}
			\item $ \displaystyle \frac{1}{n+1} \binom{2n}{n} $ linear intervals of height $ 0 $,
			\item  $\displaystyle \binom{2n-1}{n-2}$ linear intervals of height $ 1 $,
			\item  $\displaystyle 2 \binom{2n-k}{n-k-1}$ linear intervals of height $ k $, for $ 2 \leq k < n $.
		\end{itemize}
		Furthermore, there are  no linear interval of height $ k \geq n $.
		
		This adds up to $\displaystyle \frac{1}{n+1} \binom{2n}{n} + \binom{2n-1}{n-2} + 2 \binom{2n-1}{n+2}$ linear intervals in $ \Tamn $.
		
	\end{theorem}
	
	\begin{proof}
		For the intervals of height $ 0 $, this is the number of elements in $ \Tamn $.
		
		As stated in \cref{rem:maxheightlinear}, we already know that any linear interval in $ \Tamn $ is of height at most $ n-1 $.
		Thus, we can fix $ n > k > 0 $ and use the previous results to get the number of intervals of height $ k $ in $ \Tamn $.
		
		We have  $\displaystyle [t^n] \phi_k(B) = \binom{k+2+2n}{n} $. Thus, we get:
		\begin{equation*}
			[t^n] t^{k+1} \phi_k(B)  = \binom{2n-k}{n-k-1}. 
		\end{equation*}
		Now, for $ k = 1 $, we have $ S_1 = t^2 \phi_1(B) $ thus there are $\displaystyle \binom{2n-1}{n-2}$ intervals of height $ 1 $ in $ \Tamn $.
		
		For $ 2 \leq k < n$, we have  $ S_k = 2t^{k+1} \phi_k(B) $, and there are $\displaystyle 2 \binom{2n-k}{n-k-1}$ intervals of height $ k $.
		
		Finally, we have $\displaystyle \sum_{k=2}^{n-1} \binom{2n-k}{n+1} = \sum_{k=0}^{n-3} \binom{n+1+k}{n+1} = \binom{2n-1}{n+2}$.
		This can be proven combinatorially, as a particular case of the identity $\displaystyle \sum_{k=0}^b \binom{a+k}{a} = \binom{a+b+1}{b}$.
	\end{proof}

	\section{Linear intervals in the Dyck lattice} \label{sec:Dyck}
	
	\subsection{Structure of linear intervals}
	
	As in the Tamari lattice, we will first define a family of "left" and "right" intervals that will be linear, then we will show that all linear intervals are either trivial, covering relations, left or right intervals.
	A left (resp. right) interval is an interval $ [P,Q] $ where the Dyck word $ Q $ is obtained from $ P $ by changing a subword $ d^k u $ into $ u d^k $ (resp $ d u^k $ into $ u^k d $) for some $ k \geq 2 $.
	
	Such an interval is indeed linear of height $ k $ since at every step, there is only one valley that can be changed into a peak such that the path remains under $ Q $.
	Note that left and right intervals are exchanged by the mirror involution on Dyck paths.
	
	\begin{figure}[h]
		\begin{center}
			\centering
			\includegraphics[width=8cm]{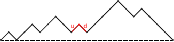}
			
			\rotatebox[origin=c]{90}{$\triangleleft$}
			
			\includegraphics[width=8cm]{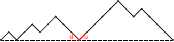}
			\caption{A covering relation in the Dyck lattice.}
			\label{fig:Dyckcoveringrelation}
		\end{center}
	\end{figure}

	\begin{proposition}
		All linear intervals of height $ k \geq 2 $ in the Dyck lattice are either left or right intervals.
	\end{proposition}
	
	\begin{proof}
		We prove the result for $ k = 2 $ and then by induction.
		
		Suppose we have a linear interval of height $ 2 $, that is of the form $ P \triangleleft  Q \triangleleft R $.
		Then the Dyck word $ Q $ is obtained from $ P $ by transforming a valley $ du $ into a peak $ ud $.
		
		If the next covering relation uses the last step $ d $ of the peak of $ Q $ we just produced, then the valley of $ P $ was followed by an up step and $ R $ is obtained from  $ P $ by changing a subword $ duu $ into $ uud $.
		Thus, $ [P,R] $ is a right interval of height $ 2 $.
		
		Similarly, if the next covering relation uses the first step $ u $ of the peak of $ Q $ we just produced, then the valley of $ P $ was preceded by a down step and $ R $ is obtained from  $ P $ by changing a subword $ ddu $ into $ udd $.
		Thus, $ [P,R] $ is a left interval of height $ 2 $.
		
		If the next covering relation happens at a valley somewhere else in $ Q $, then this valley exists in $ P $ and the two covering relations can be performed independently, thus $ [P, R] $ would be a square as shown in \cref{fig:Dycksquare} and not a linear interval. 
		
		\begin{figure}[h]
			\begin{center}
				\centering
				\begin{minipage}[c]{.45\linewidth}
					\begin{center}
						\includegraphics[width=5cm]{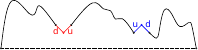}
					\end{center}
				\end{minipage}
				$ \leq $
				\begin{minipage}[c]{.45\linewidth}
					\begin{center}
						\includegraphics[width=5cm]{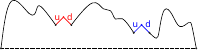}
					\end{center}
				\end{minipage}
				
				\bigbreak
				
				\begin{minipage}[c]{.45\linewidth}
					\begin{center}
						\rotatebox[origin=c]{90}{$\leq$}
					\end{center}
				\end{minipage}
				\begin{minipage}[c]{.45\linewidth}
					\begin{center}
						\rotatebox[origin=c]{90}{$\leq$}
					\end{center}
				\end{minipage}
				\bigbreak
				\begin{minipage}[c]{.45\linewidth}
					\begin{center}
						\includegraphics[width=5cm]{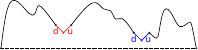}
					\end{center}
				\end{minipage}
				$ \leq $
				\begin{minipage}[c]{.45\linewidth}
					\begin{center}
						\includegraphics[width=5cm]{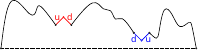}
					\end{center}
				\end{minipage}
				
				\caption{A square in the Dyck lattice.}
				\label{fig:Dycksquare}
			\end{center}
		\end{figure} 
		
		\smallskip
		Let now $ [P,Q] $ be a linear interval of height $ k+1 \geq 3 $, and $ Q' $ be the lower cover of $ Q $ in $ [P,Q] $.
		Then $ [P,Q'] $ is a linear interval of height $ k \geq 2 $.
		By induction it is either a left or right interval.
		
		\smallskip
		If $ [P,Q'] $ is a left interval, then $ Q' $ is obtained from $ P $ by changing a subword $ d^k u $ into $ u d^k $.
		Furthermore, in the chain from $ P $ to $ Q' $, the peak created in $ Q' $ is formed by the two first steps of this subword $ u d^k $, which is not followed by an up step in $ Q' $.
		Thus, the covering relation $ Q' \triangleleft Q $ has to use this first up step $ u $ of this subword $ u d^k $, as otherwise the interval would not be linear.
		Then $ Q $ is obtained from $ P $ by changing a subword $ d^{k+1} u $ into $ u d^{k+1} $ and $ [P,Q] $ is a left interval of height $ k+1 $.
		
		\smallskip
		If $ [P,Q'] $ is a right interval, then $ Q' $ is obtained from $ P $ by changing a subword $ d u^k $ into $ u^k d$.
		Symmetrically, the covering relation $ Q' \triangleleft Q $ has to use this last down step $ d $ and we get that $ [P,Q] $ is a right interval of height $ k+1 $.
	\end{proof}
	
	\subsection{Combinatorial description of linear intervals}
	
	We have described the structure of all linear intervals according to their height.
	As for the Tamari lattice, we will give a combinatorial description of them in order to write some equations.
	In fact, we will produce the very same equations as for the Tamari lattice, and thus we will prove that the two lattices have the same number of linear intervals of any fixed height.
	
	For $ k \geq 0 $, let $ T_k(t) $ be the generating series of linear intervals of height $ k $ in the Dyck lattices.
	
	\begin{proposition}
		For any $ k \geq 0 $, the generating series $ T_k(t) $ of linear intervals of height $ k $ in the Dyck lattices is equal to the generating series $ S_k(t) $ of linear intervals of height $ k $ in the Tamari lattices.
	\end{proposition}
	
	\begin{proof}
		First, let recall that the Dyck paths of size $ n $ are in bijection with the rooted planar binary trees of size $ n $.
		Thus, the generating series $ A(t) $ of binary trees is also the generating series of Dyck paths.
		Similarly, as every Dyck path of size $ n $ has $ n $ down (resp. up) steps, the generating series of Dyck paths marked at a down (resp. up) step is equal to $ t A' (t)$.
		
		As in the Tamari lattice, a linear interval of height $ 0 $ is of the form $ [P,P] $ with $ P $ a non trivial Dyck path.
		Thus, we can write:
		\begin{equation*}
			T_0 = A - 1.
		\end{equation*}
		
		Let $ [P,Q] $ be a linear interval of height $ 1 $, \textit{i.e.} a covering relation.
		Then $ Q $ is obtained from $ P $ by changing a valley $ du $ into a peak $ ud $.
		Let $ uDd $ be the excursion whose first step is the up step of this valley of $ P $.
		Then $ D $ as a word is a Dyck word.
		
		In fact, as we can see in \cref{fig:Dyckcoveringcombinatorial}, any covering relation $ P \triangleleft Q $ can be understood as a path $ P $ with a marked down step $ \# $, before which we insert $ duD $ for the bottom element $ P $ and $ udD $ for the top element $ Q $.
		We have:
		\begin{equation*}
			T_1 = tA' tA = t^2 A' A.
		\end{equation*}
		
		\begin{figure}[h]
			\begin{center}
				\centering		
				\includegraphics[width=8cm]{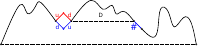}
				\caption{Decomposition of a covering relation in the Dyck lattice.}
				\label{fig:Dyckcoveringcombinatorial}
			\end{center}
		\end{figure} 
		
		\begin{figure}[h]
			\begin{center}
				\centering		
				\includegraphics[width=8cm]{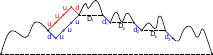}
				\caption{Decomposition of a right interval of height $ 3 $ in the Dyck lattice.}
				\label{fig:Dyckrightdecomposition}
			\end{center}
		\end{figure} 
		
		Similarly, as in \cref{fig:Dyckrightdecomposition}, let $ [P,Q] $ be a right interval of height $ k \geq 2 $. Then, $ Q $ is obtained from $ P $ by changing a subword $ du^k $ to $ u^kd $.
		Let us denote $ d_1, \dots, d_k$ the down steps of $ P $ matching with these $ k $ up steps.
		Then, in $ P $ we have a subword of the form $ du^k D_1 d_1 ... D_k d_k $, where $ D_1, \dots, D_k $ are $ k $ Dyck words, and in $ Q $, we have instead the subword  $ u^k d D_1 d_1 ... D_k d_k $.
		
		Thus, any right interval of height $ k \geq 2 $ can be understood as a Dyck path with a marked down step $ d_k $, before which we insert $ du^k D_1 d ... d D_k$ for the bottom element and $ u^k d D_1 d ... d D_k$ for the top element, with $ k $ Dyck words $ D_1, \dots, D_k $.
		
		Symmetrically, any left interval of height $ k \geq 2 $ can be understood as a Dyck path with a marked up step $ u _1$, after which we insert $  D_1 u ... u D_k d^k u$ for the bottom element and $ D_1 u ... u D_k u d^k$ for the top element, with $ k $ Dyck words $ D_1, \dots, D_k $.
		
		Thanks to this, for any $ k \geq 2 $, we finally have:
		\begin{equation*}
			T_k = tA' (tA)^k + tA' (tA)^k = 2 t^{k+1}A' A^k.
		\end{equation*}
		
		We have thus proven that for any $ k \geq 0 $, $ S_k $ and $ T_k $ satisfy the same equation, and are indeed equal.
	\end{proof}
	
	As a corollary we have the following result:
	
	\begin{theorem}
		For any $ n \geq 1 $ and $ k \geq 0 $, the Tamari lattice $ \Tamn $ and the Dyck lattice $ \Dyckn $ have the same number of linear intervals of height $ k $.
	\end{theorem}

	\section{Linear intervals in the alt-Tamari posets} \label{sec:altTamari}
	
	One can define a family of posets on the set of Dyck paths $ Z_n $ for every $ n \geq 1 $, which includes the Tamari and the Dyck lattices that we introduced before.
	We call them the alt-Tamari posets.
	More precisely, they depend on an increment function $ \delta \colon \{1, \dots, n\} \to \{0,1\}$, and will be denoted $ \Deltamn $.
	We will have a result of refinement whenever two functions $ \delta $ and $ \delta' $ are comparable.
	Moreover, in all these posets, we can define trivial intervals, covering relations, left and right intervals and prove that all linear intervals are of this kind.
	We can then give a combinatorial description of these intervals.
	It follows that we have a bijection between linear intervals of any two alt-Tamari {\color{green}posets} of the same size which preserves the height.
	This proves the main result of this part, namely the \cref{thm:maintheorem} which states that all these posets have the same number of linear intervals, even when distinguished according to their height.
	
	\subsection{Definition of the alt-Tamari posets} \label{sec:alttamaridef}
	
	Given a Dyck path of size $ n $, we will number its up steps with integers $ \{1, \dots, n\}$ increasingly from left to right.
	For example, the path $ uududdud $ will be numbered $  u_1u_2du_3ddu_4d  $.
	
	Fix some $ n \geq 1 $.
	Let $ \delta  \colon \{1, \dots, n\} \to \{0,1\} $ be an increment function.
	We introduce a notion of  \emph{$ \delta $-altitude}, where the function $ \delta $ encodes that the up step $ u_i $ increases the $ \delta $-altitude by $ \delta(i) $, while all down steps make the $\delta$-altitude decrease by $ 1 $. 
	Thus, we write $ \delta(d) = -1$ and $ \delta(u_i) = \delta(i)$. 
	
	Let $ P $ be a Dyck path of size $ n $ and $ A $ be a subword of $ P $.
	We define the \emph{$ \delta $-elevation} $ \displaystyle \delta(A) = \sum_{s \in A} \delta(s) $ as the change of $ \delta $-altitude of $ A $.

	We define the \emph{$ \delta $-excursion} of an up step $ u_i $ of $ P $ as the smallest subword $ C_i $ of $ P $ starting with $ u_i $ such that $ \delta(C_i) = 0$.
	\begin{remark} \label{rem:deltaexcursion}
		The $\delta$-excursion is well defined since the excursion $ E_i $ as defined in \cref{sec:Dyckpaths} starting at the up step $ u_i $ satisfies $ \delta(E_i) \leq 0 $.
		Moreover, $ C_i $ is always a prefix of $ E_i $.
	\end{remark}

	For instance, if $ \delta(i) = 0 $, the $ \delta $-excursion of $ u_i $ is reduced to $ u_i $.
	If $ \delta(i) = 1$ for all $ i $, then the $ \delta $-excursion of an up step $ u_i $ is exactly the excursion starting at $ u_i $.
	
	Given an increment function $ \delta $ and a Dyck path $ P $ with a valley $ d u_i $, we define the \emph{$ \delta $-rotation} of $ P $ at the up step $ u_i $ as the Dyck path $ Q $ obtained from $ P $ by exchanging the down step that precedes $ u_i $ with the $ \delta $-excursion of $ u_i $.
	In other words, if $ C_i $ is the $\delta$-excursion of $ u_i $, we can write $ P = A d C_i B $ and $ Q = A C_i d B $.
	
	\begin{figure}[h]
		\begin{center}
			\centering		
			\includegraphics[width=8cm]{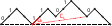}
			\caption{The $\delta$-excursion $ C_3 $ of $ u_3 $ on a Dyck path for $\delta = (0,1,1,1,0,1,0)$.}
			\label{fig:deltaexcursion}
		\end{center}
	\end{figure} 
	
	\begin{remark} \label{rem:delta-dyck}
		If $ Q $ is the $ \delta $-rotation of $ P $ at the up step $ u_i $, then $ Q $ is strictly greater than $ P $ in the Dyck lattice.
	\end{remark}
	\begin{proof}
		Let $ C_i $ be the $\delta$-excursion of $ u_i $ in P.
		We can write $ P = A d C_i B $ as a word and then we have $ Q = A C_i d B $.
		The Dyck path $ Q $ is obtained from $ P $ by moving the down step $ d $ that precedes $ u_i $ to the right.
		
		This can be achieved as a sequence of covering relations in the Dyck lattice.
		Indeed, when moving this down step $ d $ letter by letter to the right, either it is exchanged with an up step $ u_k $ and it is a covering relation in the Dyck lattice or it is exchanged with another down step and the Dyck path is unchanged.
	\end{proof}
	
	This proves that there are no cycles of $\delta$-rotations.
	We can thus define the alt-Tamari poset  $ \Deltamn $ as the reflexive transitive closure of $\delta$-rotations on the set $ Z_n $ of Dyck paths of length $ n $.
	
	\begin{remark} \noindent
		\begin{itemize}
			\item The first value $ \delta(1) $ does not matter in the definition of the poset since $ u_1 $ is never in a valley.
			\item 	The last value $ \delta(n) $ does not matter either in the definition. Indeed, the $\delta$-excursion of the last step $ u_n $ can be either $ u_n $ or $ u_n d $ and in both cases, a $\delta$-rotation at the last up step will always change $ du_nd $ into $ u_n d d $.
			\item 	We recover the \emph{Tamari lattice } when the increment function $\delta$ is such that $ \delta(i) = 1$ for all $ i $, since the $ \delta $-excursion of any up step is always its excursion.
			\item 	We recover the \emph{Dyck lattice }when the increment function $\delta$ is such that $ \delta(i) = 0$ for all $ i $, since the $ \delta $-excursion of any up step is reduced to the up step itself.
		\end{itemize}
	\end{remark}
	
	\begin{lemma} \label{lem:twoexcursions}
		Let $ C_i $ be the $ \delta $-excursion of $ u_i $ in a Dyck path $ P $.
		Let $ C_j $ be the $ \delta $-excursion of $ u_j $ in $ P $ with $ i \neq j $.
		
		Either $ C_i $ and $ C_j $ are disjoint as subwords and we write $ C_i \cap C_j = \emptyset $ or one is included in the other.
		Moreover, they do not end at the same step.
	\end{lemma}
	
	\begin{proof}
		Suppose $ i < j $.
		Suppose that $ C_i $ and $ C_j $ are not disjoint.
		Then $ u_j $ is a step of $ C_i $.
		
		Write $ C_i = A u_j B $.
		Then as $ C_i $ is the $ \delta $-excursion of $ u_i $, all its strict prefixes $ w $ satisfy $ \delta(w) > 0 $ and so we have $ \delta(A) > 0$.
		
		Moreover, as $ C_j $ is the $ \delta $-excursion of $ u_j$, we have $ \delta(W) \geq 0$ for any prefix of $ C_j $.
		Hence, $ \delta(AW) = \delta(A) + \delta(W) > 0$.
		
		It follows that $ AC_j $ is a prefix of $ C_i $, hence $ C_j $ is included in $ C_i $.
		
		We can finally notice that $ C_i $ and $ C_j $ do not end at the same step, for the same reason.
	\end{proof}
	
	Let $ P $ be a Dyck path of size $ n $ and $\delta$ an increment function.
	
	For $ 1 \leq i \leq n $ we define $ h_i(P) $ as the position of $ u_i $ in $ P $ and $ h(P) = (h_1(P), \dots, h_n(P))$.
	For instance, if $ P =  u_1u_2du_3ddu_4d$ then $ h(P) = (1,2,4,7) $.
	
	For $ 1 \leq i \leq n $, we define $ \ell_i(P ; \delta) $ as the length of the $\delta$-excursion of $ u_i $ in $ P $ and $ \ell(P ; \delta) = (\ell_1(P ; \delta), \dots, \ell_n(P ; \delta))$.
	For instance, if $ P =  u_1u_2du_3ddu_4d$ and $\delta(i) = 1$ for all $ i $, then $ \ell(P ; \delta) = (6,2,2,2) $.
	
	On the Dyck path $ P $ of \cref{fig:deltaexcursion} we have $ h(P) = (1,2,5,6,8,9,12) $ and $ \ell(P ; \delta) = (1,2,7,2,1,2,1) $.

	
	\begin{lemma} \label{lem:deltarotation}
		Let $ P $ be a Dyck path with a valley $ du_i $ and $ C_i $ be the $\delta$-excursion of $ u_i $.
		Let $ Q $ be the result of the $\delta$-rotation of $ P $ at $ u_i $.
		
		For all $ j $ such that $ u_j \in C_i$, we have $ h_j(Q) = h_j(P) - 1$ and for all other $ j $, $ h_j(Q) = h_j(P) $.
		
		Furthermore, if there exists some $ j $ such that the $\delta$-excursion of $ u_j $ in $ P $ ends with the down step $ d $ of the valley $ d u_i $, then $ \ell_j(Q ; \delta) = \ell_j(P ; \delta) + \ell_i(P ; \delta)$.
		For all other $ j $, we have $ \ell_j(Q ; \delta) =  \ell_j(P ; \delta) $.
	\end{lemma}
	
	\begin{proof}
		The statement about the $ h_j(Q) $ is immediate by definition of $\delta$-rotations.
		
		Suppose that there exists $ j $ such that the $\delta$-excursion $ C_j $ of the up step $ u_j $ in $ P $ ends with the down step of the valley $ d u_i $.
		Write $ C_j = Bd $, and $ P = A Bd C_i D $.
		Then $ Q = A B C_i d D $.
		
		The equality $ C_j = Bd$ implies $ \delta(B) = 1 $ and for all non empty prefixes $ w $ of $ B $, we have $ \delta(w) > 0$.
		Similarly, for all prefixes $ w' $ of $ C_i $, we have $\delta(w') \geq 0$.
		So for all prefixes $ w'' $ of $ B C_i $, we have $\delta(w'') > 0$ and $\delta(BC_id) = \delta(BdC_i) = 0$.
		
		Thus, $ BC_id $ is the $\delta$-excursion of $ u_j $ in $ Q $ and we have $ \ell_j(Q ; \delta) = \ell_j(P ; \delta) + \ell_i(P ; \delta)$.
		
		Suppose that $ j $ is such that the $\delta$-excursion $ C_j $ of the up step $ u_j $ in $ P $ does not end with the down step of the valley $ d u_i $.
		
		If $ C_j  \cap d C_i  = \emptyset$, then it is clear that the $\delta$-excursion of $ u_j $ in $ Q $ is still $ C_j $.
		
		If $ C_j \subseteq C_i $ then it is immediate as well that the $\delta$-excursion of $ u_j $ in $ Q $ is still $ C_j $, but translated.
		
		If $ C_i \subseteq C_j $ then we write $ C_j = A d C_i B $ in $ P $.
		The $\delta$-excursion of $ u_j $ in $ Q $ will then be $ A C_i d B $ and its length does not change either.
	\end{proof}
	
	This proves in particular that for all $ Q \geq P $ in $ \Deltamn $, $ h(Q) \leq h(P) $ and $ \ell(Q ; \delta) \geq \ell(P ; \delta) $ component-wise.
	
	\begin{remark}
		One can wander if we have the converse implication, which would give a characterization of the alt-Tamari order. 
		This converse implication holds for the Dyck and the Tamari lattices.
	\end{remark}
	
	\begin{lemma} \label{lem:stepinexcursion} 
		Let $ P $ a Dyck path such that $ u_j $ is in the $\delta$-excursion of $ u_i $.
		
		For all $ Q \geq P $ in $ \Deltamn $, $ u_j $ is in the $\delta$-excursion of $ u_i $.
	\end{lemma}
	
	\begin{proof}
		It is sufficient to prove this for all covering relations of $ P $.
		By definition, $ u_i $ is in the $\delta$-excursion of $ u_j $ in $ P $, if and only if $ h_i(P) -  h_j(P) \leq \ell_j(P)$.
		
		Suppose $ Q $ is an upper cover of $ P $.
		Because of \cref{lem:deltarotation}, we have either $ h_j(Q) = h_j (P) - 1 $ or $ h_j(Q) = h_j (P) $.
		
		\smallskip
		If $ h_j(Q) = h_j (P) - 1 $ then, the full $\delta$-excursion of $ u_j $ is moved.
		Because $ u_i $ is in the $\delta$-excursion of $ u_j $ in $ P $, we have also $ h_i(Q) = h_i (P) - 1  $.
		In this case, $ h_i(Q) -  h_j(Q) = h_i(P) -  h_j(P) $.
		
		If $ h_j(Q) = h_j(P) $ then either we have $ h_i(Q) = h_i (P) - 1 $ or $ h_i(Q) = h_i (P) $.
		Thus, $ h_i(Q) -  h_j(Q) $ is either equal to $ h_i(P) -  h_j(P) $ or to $ h_i(P) -  h_j(P) - 1$.
		
		\smallskip
		In all cases, we have $ h_i(Q) -  h_j(Q) \leq h_i(P) -  h_j(P) \leq \ell_j(P ; \delta) \leq \ell_j(Q ; \delta) $. 
		The last inequality is again guaranteed by \cref{lem:deltarotation}.
		Finally, the inequality $ h_i(Q) -  h_j(Q) \leq \ell_j(Q ; \delta) $ implies that $ u_i $ is still in the $\delta$-excursion of $ u_j $ in $ Q $.
	\end{proof}
	
	\begin{proposition} \label{prop:deltacov}
		The covering relations in the poset $ \Deltamn $ are exactly all the $ \delta $-rotations.
	\end{proposition}
	\begin{proof}
		As the poset is defined as the reflexive transitive closure of $\delta$-rotations, all covering relations are $\delta$-rotations.
		We have to prove the converse, namely that no $\delta$-rotation can be achieved as a non trivial sequence of $ \delta $-rotations.
		
		Suppose $ Q $ is the $\delta$-rotation of $ P $ at the up step $ u_i $, and $ C_i $ is the $\delta$-excursion of $ u_i $ in $ P $.
		We write $ P = A d C_i B $ and $ Q = A C_i d B $. 
		
		Then, since $ A $ and $ B $ are unchanged, no $\delta$-rotation is possible at an up step in $ A $ or in $ B $.
		
		Suppose $ P \triangleleft Q_1 \triangleleft \dots \triangleleft Q_k = Q $ is a sequence of $ \delta $-rotations from $ P $ to $ Q $.
		Then all $\delta$-rotations of the sequence have to happen at steps in $ C_i $.
		
		Because $ h_i(Q) = h_i(P) - 1$ and $ u_i $ is the only up step of $ C_i $ that changes $ h_i $, then one of those $\delta$-rotations has to happen at $ u_i $.
		
		If the first $\delta$-rotation happens at $ u_j \neq u_i $, then $ h_j(Q_1) = h_j(P) - 1 = h_j(Q) $.
		Thus, $ h_j $ can not decrease any more.
		But then, as \cref{lem:stepinexcursion} ensures that $ u_j $ will always remain in the $\delta$-excursion of $ u_i $ in all the sequence, the $\delta$-rotation at $ u_i $ would make $ h_j $ decrease by $ 1 $ and this is not possible.
		This proves that the first $\delta$-rotation has to happen at $ u_i $ and $ Q_1 = Q $.
		
		Therefore, $ P \leq Q $ is the only chain from $ P $ to $ Q $ and thus it is a covering relation.
	\end{proof}
	
	Let us now prove that there is a boolean structure of refinements of the alt-Tamari posets.
	More precisely, given two increment functions $\delta$ and $\delta'$ of size $ n $, whenever we have $ \delta \leq \delta' $ component-wise, then $ \Deltamn $ refines $ \operatorname{Tam}^{\delta'}_n$.
	This means that whenever we have two Dyck paths $ P $ and $ Q $ such that $ P \leq Q $ in $ \operatorname{Tam}^{\delta'}_n$, then $ P \leq Q $ in $ \Deltamn $. 
	
	\begin{lemma} \label{lem:twodeltas}
		Let $ P $ be a Dyck path of size $ n $ and $\delta \leq \delta'$ two increment functions.
		The $\delta$-excursion $ E_i $ of $ u_i $ in $ P $ is a prefix of the $\delta'$-excursion $ E'_i $ of $ u_i $.
	\end{lemma}
	
	\begin{proof}
		Let $ w $ be any strict prefix of $ E_i $.
		As $ E_i $ is a $\delta$-excursion, we have $ \delta(w) > 0 $.
		Then, as $ \delta' \geq \delta $, we also have $ \delta'(w) > 0 $.
		Similarly, $ \delta'(E_i) \geq \delta(E_i)  = 0$.	
	\end{proof}
	
	As the alt-Tamari lattice  $ \operatorname{Tam}^{\delta'}_n$ is defined as the transitive reflexive closure of $\delta'$-rotations, it is sufficient to prove that any $\delta'$-rotation $ P \triangleleft Q $ defines an interval $ [P,Q] $ in $ \Deltamn $.
	
	\begin{proposition}
		Let $\delta \leq \delta'$ be two increment functions of size $ n $.
		Let $ P \triangleleft Q $ be a covering relation in  $ \operatorname{Tam}^{\delta'}_n$.
		We have $ P < Q $ in $ \Deltamn $.
	\end{proposition}
	
	\begin{proof}
		Using \cref{lem:twoexcursions,lem:twodeltas}, we can build a chain from $ P $ to $ Q $ in $ \Deltamn $, just as we did to prove \cref{rem:delta-dyck}, namely that the Dyck lattice refines all alt-Tamari posets.
		
		We write $ P = A d C'_i B $ and $ Q = A C'_i d B $ with $ C'_i $ the $\delta'$-excursion of $ u_i $.
		We will exchange this down step $ d $ with down steps of $ \delta $-excursions until we reach $ Q $ and this will prove the result.
		
		The $\delta$-excursion $ C_i $ of $ u_i $ in $ P $ is a prefix of $ C'_i = C_i D$.
		Letting $ P_1 = A C_i d D B$, we have $ P \triangleleft P_1 $ in $ \Deltamn $.
		
		Now, either $ D $ starts with a down step and exchanging the two down steps does not change the path or $ D $ starts with an up step $ u_j $.
		In this second case, let $ C_j $ be the $\delta$-excursion of $ u_j $ in $ P $.
		Then $ C_j $ is a prefix of the $\delta'$-excursion $ C'_j $ of $ u_j $ in $ P $.
		Now, $ C'_j $ and $ C'_i $ are two $\delta'$-excursions whose intersection is not empty.
		\Cref{lem:twoexcursions} ensures that $ C'_j $ is thus a prefix of $ D $ and the same holds for $ C_j $.
		Then exchanging $ d $ and $ C_j $ is a $\delta$-rotation.
		
		By induction on the length of $ D $, we build a chain from $ P $ to $ Q $ in $ \Deltamn $.
	\end{proof}

	\subsection{Structure of linear intervals}
	
	Now we study the linear intervals in the alt-Tamari posets. 
	As in the Tamari lattice and the Dyck lattice, we define left and right intervals, which we prove to be linear.
	Then we prove that all linear intervals of height $ k \geq 2 $ are either left or right intervals.
	Lastly, we give a combinatorial description and deduce the same equations as previously.
	
	\begin{lemma} \label{lem:tworotations}
		Let $ \delta $ be an increment function.
		Let $ P \triangleleft Q $ be a covering relation. We can write $ P = A d C_i B $ and $ Q = A C_i d B$ with $ C_i $ the $\delta$-excursion of $ u_i $.
		
		There are at most two covering relations $ Q \triangleleft Q' $ such that $ [P,Q'] $ is a linear interval.
		Precisely, if this second covering relation is not occurring at the valley $ d u_i $ if $ A $ ends with a down step $ d $ or at the valley $ d u_k $ if $ B $ starts with the up step $ u_k $, then the interval $ [P,Q'] $ is not linear.
	\end{lemma}
	
	\begin{proof}	
		Let $ Q' $ be the $\delta$-rotation of $ Q $ at the up step $ u_j $ with $ u_j \neq u_i $ and $ u_j \neq u_k $ if $ B $ starts with the up step $ u_k $.
		
		We want to prove that $ [P,Q'] $ is not linear.
		We have one maximal chain $ P \triangleleft Q \triangleleft Q' $ from $ P $ to $ Q' $.
		We will show that there is a different maximal chain from $ P $ to $ Q' $.
		
		\medskip
		\underline{Case $ 1 $}: Suppose that $ u_j \in C_i $.
		Then, we can write $ Q' = A C'_i d B $ and if $ P' $ is the $\delta$-rotation of $ P $ at $ u_j $, we have $ P' =  A d C'_i B $, and we have another maximal chain $ P \triangleleft P' \triangleleft Q' $ from $ P $ to $ Q' $.
		Thus, $ [P,Q] $ is a non linear interval as shown in \cref{fig:deltamsquare}.
		
		\begin{figure}[h]
			\begin{center}
				\centering
				\begin{minipage}[c]{.45\linewidth}
					\begin{center}
						\includegraphics[width=5cm]{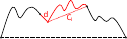}
					\end{center}
				\end{minipage}
				$ \leq $
				\begin{minipage}[c]{.45\linewidth}
					\begin{center}
						\includegraphics[width=5cm]{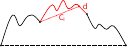}
					\end{center}
				\end{minipage}
				
				\bigbreak
				
				\begin{minipage}[c]{.45\linewidth}
					\begin{center}
						\rotatebox[origin=c]{-90}{$\leq$}
					\end{center}
				\end{minipage}
				\begin{minipage}[c]{.45\linewidth}
					\begin{center}
						\rotatebox[origin=c]{-90}{$\leq$}
					\end{center}
				\end{minipage}
				\bigbreak
				\begin{minipage}[c]{.45\linewidth}
					\begin{center}
						\includegraphics[width=5cm]{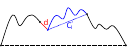}
					\end{center}
				\end{minipage}
				$ \leq $
				\begin{minipage}[c]{.45\linewidth}
					\begin{center}
						\includegraphics[width=5cm]{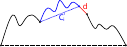}
					\end{center}
				\end{minipage}
				
				\caption{A square in an alt-Tamari poset.}
				\label{fig:deltamsquare}
			\end{center}
		\end{figure} 
		
		\medskip
		\underline{Case $ 2 $}: Suppose now that $ u_j \notin C_i $.
		Let $ C_j $ be the $\delta$-excursion of $ u_j $ in $ P $.
		Then, two more cases occur.
		Either $ C_j $ is directly followed by $ C_i $ in $ P $ or not.
		
		\medskip
		Case 2.1: Suppose that $ C_j = E d $ and $ P = A' d E d C_i B $ as in the first case of \cref{lem:deltarotation}.
		Then, we can write $ Q = A' d E C_i d B $ and $ Q' = A' E C_i d d B $.
		Set $ P' = A' E d d C_i B $ and $ P'' =  A' E d C_i d B $ and we have another maximal chain $  P \triangleleft P' \triangleleft P'' \triangleleft Q' $ from $ P $ to $ Q' $ as we can see in \cref{fig:deltampentagon}. 
		
		\begin{figure}[h]
			\begin{center}
				\centering
				\begin{minipage}[c]{.3\linewidth}
					\begin{center}
						\includegraphics[width=3.5cm]{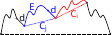}
					\end{center}
				\end{minipage}
				$ \leq $
				\begin{minipage}[c]{.3\linewidth}
					\begin{center}
						\includegraphics[width=3.5cm]{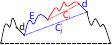}
					\end{center}
				\end{minipage}
				$ \leq $
				\begin{minipage}[c]{.3\linewidth}
					\begin{center}
						\includegraphics[width=3.5cm]{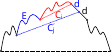}
					\end{center}
				\end{minipage}
				
				\bigbreak
				
				\begin{minipage}[c]{.45\linewidth}
					\begin{center}
						\rotatebox[origin=c]{-45}{$\leq$}
					\end{center}
				\end{minipage}
				\begin{minipage}[c]{.45\linewidth}
					\begin{center}
						\rotatebox[origin=c]{45}{$\leq$}
					\end{center}
				\end{minipage}
				
				\bigbreak
				\begin{minipage}[c]{.3\linewidth}
					\begin{center}
						\includegraphics[width=3.5cm]{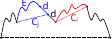}
					\end{center}
				\end{minipage}
				$ \leq $
				\begin{minipage}[c]{.3\linewidth}
					\begin{center}
						\includegraphics[width=3.5cm]{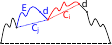}
					\end{center}
				\end{minipage}
				
				\caption{A pentagon in an alt-Tamari poset.}
				\label{fig:deltampentagon}
			\end{center}
		\end{figure}

		\medskip
		Case 2.2: Suppose that $ C_j $ is not directly followed by $ C_i $ in $ P $.
		Then the $\delta$-excursion of $ u_j $ in $ Q $ is still $ C_j $ and we can write $ Q' = A' C_i d B' $ after the $\delta$-rotation at $ u_j $.
		Let $ P' = A' d C_i B' $ be the $\delta$-rotation of $ P $ at $ u_j $.
		We have again another maximal chain $ P \triangleleft P' \triangleleft Q' $ from $ P $ to $ Q' $ and thus, a square similar to the one in \cref{fig:deltamsquare}.
		
		In all these cases, the interval $ [P, Q'] $ is not linear.
	\end{proof}
	
	\begin{definition} \label{def:altinterval}
		We say that an interval $ [P,Q] $ is a \emph{left interval} if we can write $ P = A d^k C_i B $ and $ Q = A C_i d^k B$ for some $ k \geq 2 $ with $ C_i $ a $ \delta $-excursion.
		
		We say that $ [P,Q] $ is a \emph{right interval} if we have $ P = A d C_1 \dots C_k B $ and $ Q = A C_1 \dots C_k d B$ with $ k $ $\delta $-excursions $ C_1, \dots, C_k $ for some $ k \geq 2 $.
	\end{definition}
	
	Remark that in both cases, for $ k = 1 $, these would simply describe a covering relation.
	
	\begin{proposition}
		A left interval is linear and the $ k $ that appears in \cref{def:altinterval} is its height.
	\end{proposition}
	
	\begin{proof}
		Let $ [P, Q] $ be a left interval.
		We can write $ P = A d^k C_i B $ and $ Q = A C_i d^k B$ with $ C_i $ the $\delta$-excursion of $ u_i $ for some $ k \geq 2 $.
		
		First, we clearly have a maximal chain of length $ k $ from $ P $ to $ Q $, namely $ P = P_0 \triangleleft P_1 \triangleleft \dots \triangleleft P_k = Q$, where $ P_j = A d^{k-j} C_i d^j  B$.
		We want to prove that it is the unique maximal chain from $ P $ to $ Q $.
		
		As in the proof of \cref{prop:deltacov}, for all $ u_j \notin C_i$, $ h_j(Q) = h_j(P) $ so in any chain from $ P $ to $ Q $, only $\delta$-rotations at up steps of $ C_i $ are possible.
		Moreover, for all $ u_j \in C_i $, $ h_j(Q) = h_j(P)  - k$.
		Remark than only $\delta$-rotations at $ u_i $ can make $ h(Q') $ decrease by $ 1 $ for any $ Q' \in [P,Q] $, so that any maximal chain from $ P $ to $ Q $ contains $ k $ $\delta$-rotations at $ u_i $.
		Furthermore, because of \cref{lem:stepinexcursion}, for all $ u_j \in C_i $, each of these $\delta$-rotation at $ u_i $ will make $ h_j(Q') $ decrease by $ 1 $ for all $ Q' \geq P $.
		Thus, any maximal chain from $ P $ to $ Q $ must contain exactly $ k $ $\delta$-rotations at $ u_i $ and no other $\delta$-rotations.
		
		Thus, any left interval is linear. 
	\end{proof}
	
	\begin{proposition}
		A right interval is linear and the $ k $ that appears in \cref{def:altinterval} is its height.
	\end{proposition}
	
	\begin{proof}
		Let $ [P, Q] $ be a right interval.
		We can write $ P = A d C_1 \dots C_k B $ and $ Q = A C_1 \dots C_k d B$, where $ C_1, \dots, C_k $ are $k $ $ \delta$-excursions for some $ k \geq 2 $.
		Recall that for $ k = 1 $, $ [P,Q] $ would be a covering relation and thus a linear interval of height $ 1 $.
		
		Again, we can write a maximal chain of length $ k $ from $ P $ to $ Q $, namely $ P = P_0 \triangleleft P_1 \triangleleft \dots \triangleleft P_k = Q$, where  $ P_j = A C_1 \dots C_{j} d C_{j+1} \dots C_k  B$.
		We prove by induction on $ k $ that it is the unique maximal chain from $ P $ to $ Q $.
		The case $ k = 1 $ is proven in \cref{prop:deltacov}.
		
		For all up steps $ u_j $ in one of these $ k $ $\delta$-excursions, $ h_j(Q)  = h_j(P) - 1$ and for all other $ u_j $, $ h_j(Q)  = h_j(P) $.
		Furthermore, we also have $ \ell_j(Q; \delta) = \ell_j(P ; \delta) $.
		
		Let $ u_{i_j} $ be the first step of $ C_j $.
		Because of \cref{lem:stepinexcursion}, for all $ Q' \geq P $, any $\delta$-rotation that moves $ u_{i_j} $ will also move all up steps of $ C_j $.
		Hence, in any maximal chain from $ P $ to $ Q $, the only possible $\delta$-rotations happen at the steps $ u_{i_j} $.
		
		Moreover, for all $ 1 \leq j \leq k-1$, $ C_j $ and $ C_{j+1} $ are two consecutive $\delta$-excursions in $ P $ so a $\delta$-rotation at $ u_{i_{j+1}}  $ would change $ \ell_{i_j}(P ; \delta) $ into $ \ell_{i_j}(P ; \delta) + \ell_{i_{j+1}}(P ; \delta)$, as stated in \cref{lem:deltarotation}.
		Thus, in any maximal chain from $ P $ to $ Q $, the first $\delta$-rotation must happen at $ u_{i_1} $.
		
		Then, $ [P_1, Q] $ is either a covering relation if $ k = 2 $ or a right interval with $ k-1 $ $\delta$-excursions otherwise.
		By induction, $ [P_1, Q] $ is a linear interval of height $ k-1 $ and thus, $ [P,Q] $ is an interval of height $ k $ since it contains a unique maximal chain of length $ k $.
		
		It follows that any right interval is linear.
	\end{proof}
	
	\begin{lemma} \label{lem:notlinear}
		Let $ P = A d C_1 C_2 B $ with $ C_1 = Ed $ the $\delta$-excursion of $ u_i $ and $ C_2 $ the $\delta$-excursion of $ u_j $. 
		
		If $ Q = A E C_2 dd B $, then the interval $ [P,Q] $ is not linear.
	\end{lemma}
	
	\begin{proof}
		We have two chains $ P \triangleleft A d E C_2 d B \triangleleft Q$ and $ P \triangleleft A C_1 d C_2 d B \triangleleft A C_1 C_2 d d B \triangleleft Q$ from $ P $ to $ Q $. 
		This is exactly the situation of \cref{fig:deltampentagon}.
	\end{proof}
	
	\begin{proposition}
		In the alt-Tamari poset $ \Deltamn $, all linear intervals of height $ k \geq 2 $ are either right or left intervals.
	\end{proposition}
	
	\begin{proof}
		We already know thanks to \cref{lem:tworotations} that there are only two kinds of linear intervals of height $ 2 $, and they are precisely left and right intervals. We will prove the result by induction on height.
		
		Let $ [P,Q] $ be a linear interval of height $ k+1 \geq  3 $.
		Let $ Q' $ be the lower cover of $ Q $ in $ [P,Q] $, such that $ [P,Q'] $ is a linear interval of height $ k \geq 2 $.
		By induction, it is either a left of a right interval.
		
		\medskip
		Suppose $ [P,Q'] $ is a left interval.
		Then we can write $ P = A d^k C_i B $ and $ Q = A C_i d^k B $ with $ C_i $ the $\delta$-excursion of some up step $ u_i $.
		
		As the last rotation in $  [P,Q'] $ occurs at $ u_i $ and $ C_i $ is followed by at least two down steps in $ Q' $, the \cref{lem:tworotations} assures that the only possible $\delta$-rotation $ Q' \triangleleft Q $ such that $ [P,Q] $ is a linear interval is again at the up step $ u_i $.
		This produces a left interval of height $ k+1 $.
		
		\medskip
		Suppose $ [P,Q'] $ is a right interval.
		Then we can write $ P = A d C_1 \dots C_k B $ and $ Q = A C_1 \dots C_k d B $ with $ k $ $\delta$-excursions $ C_1, \dots, C_k $.
		
		Now, the \cref{lem:tworotations} assures that there are only two up steps of $ Q' $ where a $\delta$-rotations might produce an interval $ [P,Q] $ that is still linear, namely the first step of $ C_k $ and the first step of $ B $.
		
		The \cref{lem:notlinear} shows that a $\delta$-rotation at the first step of $ C_k $ would produce a non-linear interval.
		Hence, the only possible $\delta$-rotation $ Q' \triangleleft Q $ happens at the first step of $ B $ and this produces a right interval of height $ k+1 $.
	\end{proof}
	
	\subsection{Combinatorial description and counting}

	We have described the structure of all linear intervals of the alt-Tamari posets, according to their height.
	We can adapt the combinatorial description that we gave for the Dyck lattice, and produce again a decomposition which generalizes the cases of the Dyck and the Tamari lattices.
	This proves the main result of this section, namely that all the alt-Tamari posets share the same number of linear intervals of any height.
	
	As all the alt-Tamari posets are defined on Dyck paths, it is immediate that they all have the same number of trivial intervals, or in other words, (linear) intervals of height $ 0 $.
	Furthermore, as a covering relation can be described as a Dyck path with a marked valley, it follows that every Dyck path has the same number of upper covers in every alt-Tamari posets.
	Thus, the total number of covering relations of all alt-Tamari posets is the same, namely the number of linear intervals of height $ 1 $.
	
	In this section, we will prove that for any $ k \geq 2 $, all these posets have the same number of right (resp. left) intervals of height $ k $, namely $\displaystyle \binom{2n-k}{n+1} $.
	This will be proven through a bijection between a right (resp. left) interval of height $ k $ and a Dyck path marked at a down (resp. up) step and a sequence of $ k $ Dyck paths (possibly trivial).
	The same decomposition works also for intervals of height $ k = 1 $, that is to say, they are in bijection with a pair of Dyck paths, the first being marked at a down step.
	We start with this case.
	
	In all what follows, we will use the notion of excursions as defined in \cref{sec:Dyckpaths}, not to be confused with $\delta$-excursions of an up step as defined in \cref{sec:alttamaridef}, which is always a prefix of the excursion of this up step (\cref{rem:deltaexcursion}).
	
	\begin{proposition}
		Let $\delta$ be an increment function.
		There is a bijection between covering relations in $ \Deltamn $ and pairs $ (P_0,P_1) $ of Dyck paths, where $ P_0 $ is marked at a down step, and the lengths of $ P_0 $ and $ P_1 $ add up to $ n-1 $.
	\end{proposition}
	
	\begin{proof}
		Let $ [P,Q] $ be a covering relation in $ \Deltamn $.
		We can write $ P = A d C_i B $ and $ Q = A C_i d B $, where $ C_i $ is the $\delta$-excursion of the up step $ u_i $.
		
		Let $ E $ be the excursion of $ u_i $ as defined in \cref{sec:Dyckpaths}.
		We can write $ E = C_i D $ with $ D $ a prefix of $ B = DB'$.
		Remark that $ E $ starts with $ u_i $ and ends with the down step matching with $ u_i $.
		Thus, we can write $ E = u_i E' d$ with $ E' $ a possibly empty Dyck path.
		
		Now, we can set $ P_0 = A d B' $ and $ P_1 = E'$, so that $ P_0 $ is a Dyck path marked at the down step between $ A $ and $ B' $ and $ P_1 $ is a possibly empty Dyck path.
		The lengths of $ P_0 $ and $ P_1 $ add up to $ n-1 $.
		
		\medskip
		
		Let us prove that it is a bijection.
		Let $ P_0 $ be a Dyck path of length $ n_0 \geq 1$, marked at a down step, so that we can write $ P_0 = A d B' $ with $ d $ the marked down step of $ P_0 $.
		Let $ P_1 $ be a Dyck path of length $ n_1 \geq 0 $.
		Let $ n = n_0 + n_1 + 1 $ and $\delta$ be an increment function of size $ n $.
		
		Suppose that there are $ j $ up steps in $ A $ and let $ i = j+1 $.
		We will construct $ P $ (resp. $ Q $) by inserting a Dyck path into $ P_0 $ between $ d $ and $ B' $ (resp. between $ A $ and $ B' $), so that its first up step will become the $ i $-th up step of $ P $ and $ Q $.
		
		Let $ E = u_i P_1 d $, where the up steps of $ P_1 $ are relabelled starting with $ u_{i+1} $.
		Obviously, $ E $ is the excursion of $ u_i $.
		
		Let then $ C_i $ be the $\delta$-excursion of $ u_i $ in $ E $, where $ u_i $ is considered to be the $ i $-th up step.
		This is possible because a $\delta$-excursion is always a prefix of the excursion.
		We can write $ E = C_i D $ and $ B = D B' $.
		
		Then, setting $ P = A d C_i B$ and $ Q = A C_i d B $, we have constructed a covering relation $ P \triangleleft Q $ in $ \Deltamn $.
		It is clear that this is the reciprocal of the decomposition described above and thus, this is a bijection.
		Moreover, we have $ P = A d E B' $ and this writing does not depend on $\delta$.
	\end{proof}
	
	\begin{corollary}
		Let $\delta$ and $\delta'$ be two increment functions of the same size $ n $.
		There is a bijection between covering relations in $ \Deltamn $ and $ \operatorname{Tam}^{\delta'}_n $.
		Moreover, this bijection preserves the bottom elements of the covering relations.
		
		In particular, there are $ \displaystyle \binom{2n - 1}{n-2} $ covering relations in $ \Deltamn $ for any $\delta$.
	\end{corollary}
	
	\begin{proposition}
		Let $\delta$ be an increment function and $ k \geq 2 $.
		
		There is a bijection between left intervals of height $ k $ in $ \Deltamn $ and sequences $ (P_0,P_1, \dots, P_k) $ of Dyck paths, where $ P_0 $ is marked at an up step, and the lengths of $ P_0, \dots, P_k $ add up to $ n-k $.
	\end{proposition}
	
	\begin{proof}
		Let $ [P,Q] $ be a left interval of height $ k \geq 2 $ in $ \Deltamn $.
		By definition of left intervals, we can write $ P = A d^k C_i B $ and $ Q = A C_i d^k B $, where $ C_i $ is the $\delta$-excursion of the up step $ u_i $.
		
		Let $ u_{i_1}, \dots, u_{i_k} $ be the up steps of $ A $ matching with the $ k $ down steps of $ P $ between $ A $ and $ B $.
		Then we can write $ P = A' u_{i_1} P_1 u_{i_2} \dots u_{i_k} P_k d^k C_i B$, where $ P_1, \dots, P_k $ are $ k $ possibly trivial Dyck paths.
		We then set $ P_0 =  A' C_i B$, and it is a Dyck path marked at the first step of $ C_i $.
		It is clear that the total number of up steps in $ P_0, \dots, P_k $ is $ n - k $ since all up steps of $ P $ except $ u_{i_1}, \dots, u_{i_k} $ are in exactly one of these Dyck paths $ P_0, \dots, P_k $.
		
		\medskip
		
		Let us prove that this decomposition is bijective by constructing its reciprocal.
		
		Let $ P_0 $ be a Dyck path of length $ n_0 \geq 1$, marked at an up step, so that we can write $ P_0 = A' u B' $ with $ u $ the marked up step of $ P_0 $.
		Let $ P_1, \dots, P_k $ be $ k $ Dyck paths of respective lengths $ n_1, \dots, n_k \geq 0 $.
		Let $ n = \sum_{i=0}^{k} n_i + k $ and $\delta$ be an increment function of size $ n $.
		
		We will build $ P $ and $ Q $ by inserting these $ k $ paths and $ k $ additional pairs of up and down steps between $ A' $ and $ B' $.
		Let $ i_1 - 1 $ be the number of up steps in $ A' $ and for $ 2 \leq j \leq k $ let $ i_j = i_{j-1} + n_{j-1} + 1 $.
		
		Then, we can set $ A =  A' u_{i_1} P_1 u_{i_2} \dots u_{i_k} P_k $.
		We set as well $ P = A d^k u B' $ and this is by construction a Dyck path of length $ n $.
		Remark that the construction of $ P $ does not depend on $\delta$.
		
		Let now $ C $ be the $\delta$-excursion of the up step $ u $ preceding $ B' $ in $ P $.
		We can write $ u B' = C B $ so that $ P = A d^k C B $.
		Setting finally $ Q = A C d^k B $, we obtain a left interval $ [P,Q] $ of height $ k $.
		Again, this is clearly the reciprocal of the decomposition of left intervals defined above.
	\end{proof}
	
	\begin{corollary}
		Let $\delta$ and $\delta'$ be two increment functions of the same size $ n $ and $ k \geq 2 $.
		There is a bijection between left intervals of height $ k $ in $ \Deltamn $ and $ \operatorname{Tam}^{\delta'}_n $.
		Moreover, this bijection preserves the bottom element of the intervals.
		
		\noindent In particular, there are $ \displaystyle \binom{2n - k}{n+1} $ left intervals of height $ k $ in $ \Deltamn $ for any $\delta$.
	\end{corollary}
	
	\begin{proposition}
		Let $\delta$ be an increment function and $ k \geq 2 $.
		
		There is a bijection between right intervals of height $ k $ in $ \Deltamn $ and sequences $ (P_0,P_1, \dots, P_k) $ of Dyck paths, where $ P_0 $ is marked at a down step, and the lengths of $ P_0, \dots, P_k $ add up to $ n-k $.
	\end{proposition}
	
	\begin{proof}
		Let $ [P,Q] $ be a right interval of height $ k \geq 2 $ in $ \Deltamn $.
		By definition of right intervals, we can write $ P = A d C_1 \dots C_k B $ and $ Q = A C_1 \dots C_k d B$ with $ k $ $\delta $-excursions $ C_1, \dots, C_k $.
		
		Let $ u_{i_1}, \dots, u_{i_k} $ be the steps with which $ C_1, \dots, C_k $ start respectively.
		Let $ d_1, \dots, d_k $ be the down steps matching with $ u_{i_1}, \dots, u_{i_k} $.
		
		Let $ D_k $ be the part of the path $ P $ starting just after $ C_k $ and ending just after $ d_k $.
		Then, $ E_k = C_k D_k $ starts with $ u_{i_k} $ and ends with its matching down step $ d_k $, so $ E_k $ is the excursion of  $ u_{i_k} $.
		In particular, if $ E = C_k $ then $ D_k $ is empty.
		
		For $ k > j \geq 1$, let $ D_j $ be the part of $ P $ starting just after $ D_{j+1} $ (which may be empty) and ending just after $ d_j $.
		As previously, the excursion of  $ u_{i_j} $ in $ P $ is either $ E_j = C_j $ and $ D_j $ is empty or it is $ E_j = C_j E_{j-1} D_j$ and $ D_j $ is not empty.
		
		Now, we can write $ P = A d C_1 \dots C_k D_k \dots  D_1 B'$.
		Set $ P_0 = A d B'$, marked at the down step between $ A $ and $ B' $ and $ C_j D_j = u_{i_j} P_j d_j$ for $ 1 \leq j \leq k $ and we have the decomposition as stated.
		
		\medskip
		
		We now prove that this decomposition is bijective by constructing its reciprocal.
		
		Let $ P_0 $ be a Dyck path of length $ n_0 \geq 1$, marked at a down step, so that we can write $ P_0 = A d B' $ with $ d $ the marked up step of $ P_0 $.
		Let $ P_1, \dots, P_k $ be $ k $ Dyck paths of respective lengths $ n_1, \dots, n_k \geq 0 $.
		Let $ n = \sum_{i=0}^{k} n_i + k $ and $\delta$ be an increment function of size $ n $.
		
		We will build $ P $ and $ Q $ by inserting these $ k $ paths one by one and $ k $ additional pairs of up and down steps between $ A $ and $ B' $.
		
		Let $ i_1-1 $ be the number of up steps in $ A $. 
		Let $ C_1 $ be the $\delta$-excursion of the first step of $ u P_1 d$ where the steps are relabelled starting with $ i_1 $.
		For $ 2 \leq j \leq k $, let $ i_j - 1 $ be the number of up steps in $ A C_1 \dots C_{j-1} $ and $ C_j $ be the $\delta$-excursion of  $ u P_j d$ where the steps are relabelled starting with $ i_j $.
		Thus, for all $ 1 \leq j \leq k $, we can write $ u P_j d = C_j D_j$.
		
		Now, let $ P = A d C_1 \dots C_k D_k \dots D_1 B'$.
		Writing $ B = D_k \dots D_1 B' $, we can set $ Q = A C_1 \dots C_k d B $.
		By construction, $ [P,Q] $ is a right interval of height $ k $ and again, this is clearly the reciprocal of the decomposition defined above.
	\end{proof}
	
	\begin{corollary}
		Let $\delta$ and $\delta'$ be two increment functions of the same size $ n $ and $ k \geq 2 $.
		There is a bijection between right intervals of height $ k $ in $ \Deltamn $ and $ \operatorname{Tam}^{\delta'}_n $.
		
		\noindent In particular, there are $ \displaystyle \binom{2n - k}{n+1} $ right intervals of height $ k $ in $ \Deltamn $ for any $\delta$.
	\end{corollary}
	
	\appendix
	
	\section{Table}
	
	Here are the first few coefficients of the series $ S_k $.  The total sequence can be found on the OEIS \cite[A344136]{oeis:A344136}.
	
	\begin{table}[h!]
		\begin{tabular}{c|ccccccc|c}
			$t^n \backslash S_k$ & $S_0$                         & $S_1$               & $S_2$                 & $S_3$                 & $S_4$                   & $S_5$ & $S_6$ & \textrm{Total} \\ \hline
			$t^1$                & 1                             &                     &                       &                       &                         &       &    &  1 \\
			$t^2$                & 2                             & 1                   &                       &                       &                         &       &     &  3 \\
			$t^3$                & 5                             & 5                   & 2                     &                       &                         &       &     &  12\\
			$t^4$                & 14                            & 21                  & 12                    & 2                     &                         &       &    &  49 \\
			$t^5$                & 42                            & 84                  & 56                    & 14                    & 2                       &       &    &  198 \\
			$t^6$                & 132                           & 330                 & 240                   & 72                    & 16                      & 2     &  &   792  \\
			$t^7$                & 429                           & 1287                & 990                   & 330                   & 90                      & 18    & 2  &   3146\\     
		\end{tabular}
	\end{table}

	\bibliographystyle{plain}
	\bibliography{article_TamLin.bib}
	
	\noindent Clément Chenevière\\
	
	\noindent 	Institut de Recherche Mathématique Avancée, UMR 7501\\
	Université de Strasbourg et CNRS\\
	7 rue René-Descartes, 67000 Strasbourg, France\\
	\&\\
	\noindent Ruhr-Universität Bochum,\\
	Universitätsstraße 150, 44780 Bochum, Germany\\
	
	\noindent E-mail: ccheneviere@unistra.fr \\
	Website: \url{https://irma.math.unistra.fr/~cheneviere/}

\end{document}